\documentclass[11pt]{article}

\usepackage{float}
\usepackage{pstricks}
\usepackage{pstricks-add}
\usepackage{epic,eepic}
\usepackage{pst-math,pst-plot}
\usepackage{pst-grad, pst-text, pst-node, pst-tree}
\usepackage{enumerate,graphicx,psfrag}
\usepackage{amsthm,amsfonts,amsmath}
\usepackage{a4wide,psfrag}

\newtheorem{proposition}{Proposition}[section]
\newtheorem{assumption}{Assumption}[section]
\newtheorem{theorem}[proposition]{Theorem}
\newtheorem{lemma}[proposition]{Lemma}
\newtheorem{corollary}[proposition]{Corollary}

\theoremstyle{remark}
\newtheorem{example}[proposition]{Example}
\newtheorem{remark}[proposition]{Remark}

\numberwithin{equation}{section}

\renewcommand{\L}[1]{\mathbf{L^{\boldsymbol{#1}}}}

\newcommand{\reali}{{\mathbb{R}}}

\newcommand{\Caption}[1]{\caption{\small#1}}
\newcommand{\dhk}{{D_{hk}}}
\newcommand{\dss}{D_{ss}}
\newcommand{\adm}{\eta} 
\newcommand{\stat}{\varphi} 
\newcommand{\Ta}{T_1^\varepsilon}
\newcommand{\Tb}{T_2^\varepsilon}
\newcommand{\cTa}{\mathcal{T}_1^\varepsilon}
\newcommand{\cTb}{\mathcal{T}_2^\varepsilon}

\newcommand{\dd}{\; \mathrm{d}}
\newcommand{\ds}{\dd s}
\newcommand{\dv}{\dd v}
\newcommand{\dy}{\dd y}

\begin{document}
\title{Existence and Stability of Traveling Waves for an Integro-differential Equation 
for Slow Erosion}
\author{Graziano Guerra$^1$ and Wen Shen$^2$}

\date{}

\footnotetext[1]{Dept.~of Mathematics and Applications,
  Milano-Bicocca University, Italy\quad \texttt{graziano.guerra@unimib.it}}

\footnotetext[2]{Dept.~of Mathematics, Penn State University,
  U.S.A.  \texttt{shen\_w@math.psu.edu}}
\maketitle


\begin{abstract}
We study an integro-differential equation that describes 
the slow erosion of granular flow.
The equation is a first order non-linear conservation law
where the flux function includes an integral term.
We show that there exist unique traveling wave solutions
that connect profiles with equilibrium slope at $\pm\infty$. 
Such traveling waves take very different forms {from} those in 
standard conservation laws. 
Furthermore, we prove that the traveling wave profiles are locally stable,
i.e., solutions with  monotone initial data
approaches the traveling waves asymptotically as $t\to+\infty$.
\end{abstract}

\textbf{keywords:}
traveling waves, existence and stability, integro-differential equation, conservation law.

\section{Introduction}

We consider the Cauchy problem for the scalar integro-differential equation 
\begin{equation}\label{1.1}
u_t(t,x) - \left(\exp\int_x^\infty f(u_x(t,y))\dy \right)_x =0\,,
\qquad 
u(0,x) = \bar u (x).
\end{equation}
The model describes the slow erosion of granular flow,
where $u(t,x) $ is the height of the standing profile of granular matter.
We assume that the slope of the profile has a fixed sign, i.e., $u_x>0$,
and granular matter is poured at a constant rate {from} an uphill location
outside the interval of interest, and slides down the hill as a very thin layer.
The interaction between the two layers is controlled by the erosion function
$f$, which denote the rate of the mass being eroded (or deposited if negative)
per unit length travelled in $x$ direction per unit mass passing through.
We assume that the erosion rate $f$ depends only on the slope $u_x$. 
At the critical slope $1$ (normalized), we have $f(1)=0$.
If $u_x>1$, we have erosion and $f>0$. 
Otherwise, if $u_x<1$, we have deposition and $f<0$.
The independent time-variable $t$ denotes the amount of mass that
has passed through, in a very long time.
We will still refer to $t$ as ``time'' throughout the paper, and call 
$\bar u(x)$ the ``initial data''.

The model was first derived  in \cite{AS-arma}
as the slow erosion limit of a $2\times2$ model for granular flow 
proposed by Hadeler \& Kuttler in \cite{HK}, 
with a specific erosion function $f(u_x)=(u_x-1)/u_x$. 
Later on, more general classes of erosion functions were studied,
making distinction between whether the slope $u_x$ blows up or remains uniformly
bounded. 
Let $w\dot=u_x$ denote the slope. 
Under the following assumptions on the erosion function $f\in\mathcal{C}^2$
\begin{equation}
  \label{f-prop}
  f(1)=0, \qquad f' \ge 0, \qquad f'' \le 0
\end{equation}
and 
\begin{equation}\label{fp1} 
\lim_{w\to+\infty} \frac{f(w)}{w}=0,
\end{equation}
the slope $w$ remains uniformly 
bounded for all $t\ge0$, see \cite{AS-ft}.
In this case, one could study the following conservation law for $w$,
\begin{equation}
  \label{1.2}
  w_t + \left(f(w) \cdot \exp\int_x^\infty f(w(t,y))\dy \right)_x ~=~0.
\end{equation}
Here the flux contains an integral term in $x$. 
Due to the nonlinearity of the function $f(w)$, 
jumps in $w$ could develop in finite  
time even for smooth initial data, which leads to kinks in the profile $u$.
Thanks to the uniform bound on $w$, global
existence and uniqueness of BV solutions for \eqref{1.2}  are established in
\cite{AS-ft,AS-se}.

However, if we allow more erosion for large slope, the solutions behave
very differently. 
If the erosion function  approaches a linear function for large $w$, 
i.e, if \eqref{fp1} is replaced by
\begin{equation}
  \label{fprop2}
  \lim_{w\to +\infty} f(w) - w f'(w) < \infty,
\end{equation}
then the slope $w$ could blow up, leading to vertical jumps in the profile,
and $w=u_x$ would contain point masses.
In this case one must study the equation \eqref{1.1}.
It is observed in \cite{ShZh} that 3 types of singularities may 
occur in the solutions of  \eqref{1.1}, namely
\begin{itemize}
\item a kink, where $u_x$ is discontinuous;
\item a jump, where $u$ is discontinuous;
\item a hyper-kink, where $u$ is continuous, but the right limit of $u_x$ is infinite, 
or both left and right limits of $u_x$ are infinite. 
\end{itemize}
The global existence of BV solutions for \eqref{1.1}
is obtained in \cite{ShZh},
through a modified version of front tracking algorithm which
generates piecewise affine approximations that also allow discontinuities. 

We remark that, model \eqref{1.1} differs {from} 
other integro-differential equations in the literature 
where the gradient $u_x$ may blow up. 
For example, 
the variational wave equation \cite{BPZ,HZ}
and the Camassa-Holm equation \cite{CH} are both well-studied.
In both cases, 
thanks 
to an a-priori bound on the $\mathbf{L}^2$ norm of $u_x^2$,
the solution $u$ 
remains H\"{o}lder continuous at all times.
In contrast,  the solutions to our equation \eqref{1.1} could develop jumps,
and the distributional derivative $u_x$ could contain point masses.

Since $u(t,x)$ is an increasing function in $x$, 
the inverse $X(t,u)$ is well-defined.
We define the corresponding gradient  
\begin{equation*}z(t,u) ~\dot=~ X_u(t,u). \end{equation*}
Formal computation shows that $X(t,u)$ and $z(t,u)$ 
are conserved quantities, and they satisfy the equations
\begin{eqnarray}
    \label{1.2x}
    X_t + \left( \exp\int_u^\infty g(z(t,v))\dv \right)_u &=& 0\,,\\
    \label{1.2z}
    z_t - \left(g(z) \cdot \exp\int_u^\infty g(z(t,v))\dv \right)_u &=&0\,.
\end{eqnarray}
Here 
\begin{equation}\label{gdef}
g(z) ~\dot= ~z f(1/z)
\end{equation} 
is the erosion function in the coordinates $(t,u)$,
representing the rate of erosion per unit mass passed per
unit distance in $u$.  
{From} the properties of $f$ in \eqref{f-prop} and \eqref{fprop2},
the function $g(z)\in \mathcal{C}^2$ satisfies
\begin{equation}\label{gprop}
g(1)=0, \qquad g(0) \ge 0, \qquad g''<0.
\end{equation}

Note that, for a given $t$, when $z(t,u)=0$ on an interval in $u$, 
the physical slope $w(t,x)$ blows up to infinity, and the profile
$u(t,x)$ has a vertical jump. 
However, the solutions for \eqref{1.2z} could become negative, 
which have no physical meaning. 
Therefore equation  \eqref{1.2z} must be equipped with the pointwise constraint $z\ge0$.
We now modify equation \eqref{1.2z} into
\begin{equation}\label{1.3}
z_t - \left(g(z) \cdot \exp\int_u^\infty g(z(t,v))\dv \right)_u ~=~\mu.
\end{equation}
The measure $\mu$ in \eqref{1.3} yields the projection into the cone of 
non-negative functions. 
Therefore, we ``transformed'' the point mass in $u_x(t,x)$ into 
constraint in $z(t,u)$. 
It turns out that the projection reduces the $\mathbf{L}^1$ distance 
between solutions $z(t,u)$. 
Thanks to this property, in \cite{CGS} we proved continuous dependence 
on initial data and on the erosion function, for the entropy weak solutions 
generated as the limit of a front tracking approximation.
This establishes a Lipschitz semigroup for the solutions of \eqref{1.3}.

\medskip

In this paper we are interested in the traveling wave solutions 
of \eqref{1.1}.
We seek a traveling wave that connects profiles $u(t,x)$ with slope $w=1$ at both 
$-\infty$ and $+\infty$, with slope $w>1$ in between, 
traveling with speed $\sigma$, i.e.,
\begin{equation}\label{2.1}
w(t,x)=W(\xi), \qquad \xi = x-\sigma t, \qquad \lim_{\xi\to \pm\infty}W(\xi)=1, \qquad
W(\xi)\ge 1. 
\end{equation}
In the variable $u(t,x)$, these become profiles
that travel upwards along  lines of slope 1 with constant speed.  
See Figure~\ref{fig1} for an illustration. 

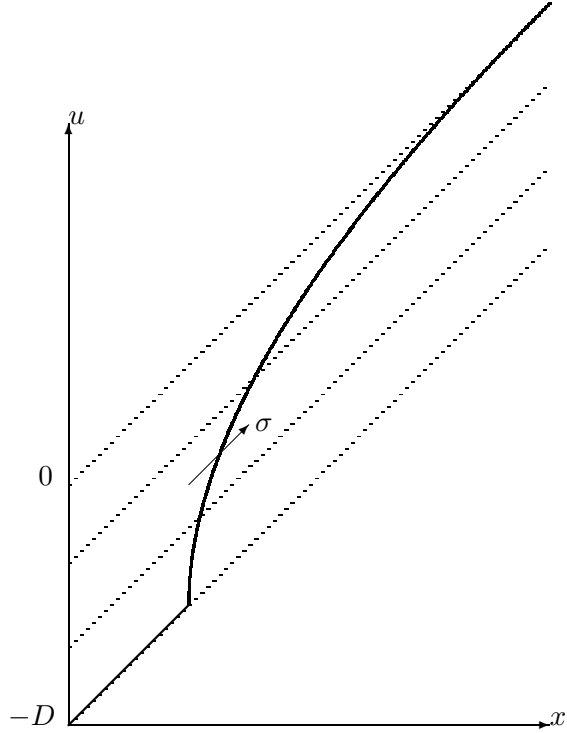
\begin{figure}[htbp]
\begin{center}
\setlength{\unitlength}{0.8mm}
\linethickness{0.1mm}
\begin{picture}(80,120)(-5,0)
\put(0,0){\vector(1,0){80}$x$}
\put(0,0){\vector(0,1){100}}\put(0,100){$u$}
\multiput(0,0)(1,1){80}{\line(1,0){0.5}}
\multiput(0,40)(1,1){80}{\line(1,0){0.5}}
\multiput(0,13)(1,1){80}{\line(1,0){0.5}}
\multiput(0,27)(1,1){80}{\line(1,0){0.5}}
\put(20,40){\vector(1,1){10}}\put(31,49){$\sigma$}
\put(-5,40){$0$}
\put(-10,0){$-D$}
\thicklines
\put(0,0){\line(1,1){20}}
\qbezier(20,20)(20,60)(80,120)
\end{picture}
\caption{Example of a traveling wave profile for $u(t,x)$.}
\label{fig1}
\end{center}
\end{figure}

We now make an important observation.   In Figure \ref{fig1} we see that, 
viewed in the direction of the dotted lines with slope 1, 
the profile remains stationary in $t$. 
This motivates another variable change which would yield
stationary traveling waves.

To this end, 
we consider solutions $z(t,u)$ to \eqref{1.3} such that, for all finite $t \ge 0$, 
\begin{equation}\label{cond-z}
z(t,\cdot)\in BV\,, \qquad 
z(t,\cdot)-1 \in \L1\,,  \qquad
\lim_{u\to\pm\infty} z(t,u) =1,\qquad 0\le z(t,u)\le 1. 
\end{equation}
This indicates that the profile $u(t,x)$ approaches a linear asymptote with slope
1 at both $x=\pm\infty$.
Without loss of generality, we assume that $u(t,x)$ approaches $u=x$ 
at $x\to +\infty$. 

We define the ``drop'' function
\begin{equation}\label{defq}
q(t,u) ~\dot =~ \int_u^{+\infty} (z(t,v)-1)\dv.
\end{equation}
Note that $q(t,u)$ indicates
the vertical drop of the profile $u(t,x)$ at $(t,u)$ 
comparing to the line $u=x$.
Under our assumptions \eqref{cond-z}, $q(t,u)$ is an increasing function in $u$, and 
approaches 0 as $u\to +\infty$. 
Therefore, $q(t,u) \le 0$.

We also define the ``total drop'' of the profile as
\begin{equation}\label{TD}
D ~\dot =~ - \int_{-\infty}^{+\infty} (z(t,v)-1)\dv = \|z(t,\cdot)-1\|_{\L1}.
\end{equation} 
This denotes the vertical drop between the lines of asymptote 
at $x\to \pm\infty$. See Figure~\ref{fig1}.
We see that, as $x\to-\infty$, the profile $u(t,x)$ approaches the
asymptote $u=x-D$. 

We remark that $z=1$ is a trivial solution.
Under the assumptions \eqref{cond-z}, the $\L1$ norm of $z-1$ 
remains constant in $t$. 
Therefore the total drop $D$ also remains constant in $t$.

If we assume further that 
$z(t,u)=1$ outside an interval $I\subset\reali$, while $z(t,u)<1$ on
$I$, then the function $u\mapsto q(t,u)$ is invertible on $I$. Let
$\tilde u(t,q)$ be its inverse defined for $q\in\left(-D,0\right)$. 
We now consider $(t,q)$ as the independent  variables,
and define the composite functions
\begin{equation}\label{eq:defs}
{\mathcal X}(t,q) = X\left(t,\tilde u(t,q)\right), \qquad
\zeta(t,q) = z\left(t,\tilde u(t,q)\right). 
\end{equation}
In this new coordinate, the quantities 
${\mathcal X}(t,q)$ and ${\mathcal X}_q(t,q)$ are the conserved.
We define the corresponding  erosion function,
\begin{equation}\label{hDef} 
h(z)~ \dot=~ \frac{g(z)}{1-z}=\frac{f(1/z)}{1/z-1} \quad (0\le z<1), \qquad h(1)=-g'(1)=f'(1).
\end{equation}

For smooth solutions, we formally have
\begin{eqnarray}
{\mathcal X}_t(t,q) + \left( \exp\int_q^0 h(\zeta(t,s))\ds \right)_q &=& 0\,, \label{eq:xi0}
\\
({\mathcal X}_q)_t - \left( h(\zeta) \cdot \exp\int_q^0 h(\zeta(t,s))\ds \right)_q &=& 0\,.
 \label{eq:xiq0}
\end{eqnarray} 

Treating $\zeta(t,q)$ as the unknown, and using the identity 
$ {\mathcal X}_q = \zeta/(1-\zeta)$, 
equation \eqref{eq:xiq0} can be rewritten as
\begin{equation}
  \label{eq:zeta0}
  \zeta_t - (1-\zeta)^2 \left( h(\zeta) \cdot \exp\int_q^0 h(\zeta(t,s))\ds \right)_q 
  = 0.
\end{equation} 


By this construction, smooth traveling wave solutions are stationary solutions 
of \eqref{eq:zeta0}. 
This is the main technique in our analysis. 
We will construct traveling waves as
stationary solutions of $\zeta(t,q)=Z(q)$ for \eqref{eq:zeta0}.

Depending on the size of the total drop $D$, 
different types of profiles can be constructed,
and different types of singularities will form in the solutions. 
In particular, 
all the stationary profiles have a downward jump at $q=-D$ (kink for the profile $u$), 
with possibly an interval where $Z=0$ (a shock  in the profile $u$), 
and then possibly a smooth stationary rarefaction fan. 
In all cases, $Z(q)$ is non-decreasing on $-D<q\le0$. 
See Section 3 for details.
We will also 
show that such profiles are unique with respect to the total drop $D$. 
For any given $D$, there exist a unique traveling wave profile.

The corresponding result for the physical variables $u(t,x)$ and $w(t,x)$
follows {from} the well-posedness of the variable changes \eqref{eq:defs}.

\bigskip

For the Cauchy problem of \eqref{1.3} with initial data satisfying \eqref{cond-z}, 
the existence of a Lipschitz semigroup of BV solutions is established in 
\cite{CGS}.
In turn this result provides us also the existence of semi-group solutions for 
the  new  variable $\zeta(t,q)$.

We will study the local stability of the stationary traveling wave profiles. 
We show that, if the initial data is ``non-decreasing'', the solution 
approaches a traveling wave profile $Z(q)$ as $t\to+\infty$.
By ``non-decreasing'', with a slight abuse of notation, we mean initial data which
satisfies the following assumptions
\begin{equation}\label{z02}
  z(0,u) = \begin{cases}
  1 \quad &  (u<u_a),\\
  \tilde z_o(u)\quad & (u\ge u_a), \end{cases}
  \qquad \tilde z_o(u_2)-\tilde z_o(u_1)\ge0\quad \text{for}\quad
  u_2\ge u_1\ge u_{a}\,.
\end{equation}
Note that this property is shared by the traveling wave profile. 
As we will see later in Section 4, property \eqref{z02} 
will be preserved in the solution for all $t>0$. 
We will prove that, 
solutions of the Cauchy problem for \eqref{1.3} 
with initial data satisfying \eqref{cond-z} and \eqref{z02},
converges to the traveling wave profile asymptotically.
Details will be explained in Section 4.

\bigskip

The rest of the paper is organized as follows. 
In section 2 we make some basic analysis, where we derive waves speeds 
and prove some technical Lemmas.
In section 3 we show 
the existence of traveling wave solutions by construction.
Furthermore, such profiles are unique with respect to the total drop.
We then return to the original variable and state the 
corresponding results in the physical variables $u(t,x),w(t,x)$.
In section 4 we establish local stability of these traveling waves, 
showing that solutions with non-decreasing initial data
approach traveling waves asymptotically as $t\to+\infty$.
A numerical simulation is given in Section 5 to demonstrate the convergence.
Finally, we give several concluding remarks  in Section 6.

\section{Basic analysis}

\subsection{Smooth stationary solutions for $\zeta(t,q)$}

We start with the discussion on the properties of the erosion function 
$h(z)$ defined in \eqref{hDef}.

\begin{lemma}\label{lm:hprop}
For $0\le z \le 1$, the erosion function $h(z)$ satisfies  
\begin{equation}\label{hps}
h(z) \ge 0\,, \qquad 
h'(z) >0 \,, \qquad 
h''(z) < \frac{2h'(z)}{1-z}.
\end{equation}
\end{lemma}

\begin{proof}
By the definition \eqref{hDef}, we have $h(0)=g(0)\ge0$.
If $0<z<1$, since $g(z)>0$, we have $h(z) >0$. 
For $z=1$, we have $h(1) = -g'(1) >0$. 
This proves that $h(z)\ge 0$ and $h(z)=0$ if and only if  when $z=0$ and $g(0)=0$. 

For $h'$, we have 
\begin{equation}\label{dh}
h'(z) =\frac{(1-z) g'(z) + g(z)}{(1-z)^2}.
\end{equation}
Since 
\begin{equation}\label{dh2}
\frac{d}{dz} \left\{ (1-z) g'(z) + g(z)\right\} = (1-z) g''(z)  <0, \qquad 
\left\{ (1-z) g'(z) + g(z)\right\}_{z=1} =0,
\end{equation}
we have $h'(z) >0$ for $0\le z <1$.  
For $z=1$, we have 
\begin{equation*}
h'(1) = \lim_{z\to 1} \frac{(1-z) g'(z) + g(z)}{(1-z)^2} 
= \lim_{z\to 1} \frac{g'(z) + h(z)}{1-z} 
= - \lim_{z\to 1} [g''(z) + h'(z) ] = - g''(1) - h'(1),
\end{equation*}
so $h'(1) = - \frac12 g''(1) >0$, proving the second inequality in \eqref{hps}. 

By \eqref{dh} and \eqref{dh2}, we now get 
\begin{equation}\label{hp5} 
  \frac{d}{dz} \left\{(1-z)^2 h'(z) \right\} <0\,.
\end{equation}
Working out the derivative, we get
\begin{equation*}
\frac{d}{dz} \left\{ (1-z)^2 h'(z) \right\} = -2 (1-z) h'(z) + (1-z)^2 h''(z) <0\,,
\end{equation*}
proving the third property in \eqref{hps}.
\end{proof}

For notation simplicity, in the rest of this paper 
we let $F(\zeta;q)$ denote the integral term
\begin{equation}\label{Fdef}
 F(\zeta;q) ~\dot=~ \exp\int_q^0 h(\zeta(t,s))\ds\,, \qquad 
 F_q (\zeta;q) = - h(\zeta) \cdot F(\zeta;q)\,.
\end{equation}
Equation \eqref{eq:zeta0} can be rewritten as
\begin{equation}\label{eq:zetaNN}
 \zeta_t - (1-\zeta)^2 \left( h'(\zeta)\zeta_q - h^2(\zeta) \right) 
\cdot F(\zeta;q) =0.
\end{equation}
For smooth solutions, along lines of characteristics $t\mapsto q$ we have
\begin{eqnarray} 
\dot q(t) &=& -  (1-\zeta)^2 h'(\zeta) \cdot F(\zeta;q)\,,
\label{char1}\\
\dot \zeta(t,q(t)) 
&=& -(1-\zeta)^2 h^2(\zeta) \cdot F(\zeta;q)\,.
\label{char2}
\end{eqnarray}
We observe that $\dot q(t)<0$ and $\dot \zeta(t,q(t))< 0 $,
therefore  all characteristics travel to the left,
and $\zeta$ is decreasing along characteristics.

\bigskip

We now derive the ODE satisfied by smooth stationary solutions for
\eqref{eq:zetaNN}. 
Let $\tilde\phi(q)$ be a smooth stationary solution of 
\eqref{eq:zetaNN}, then it must be the solution of the Cauchy problem
\begin{equation}  \label{phi}
  \tilde\phi'(q)   =\frac{h^2(\tilde\phi)}{h'(\tilde\phi)}, \qquad 
  \tilde\phi(0)=1,    \qquad (q\le 0).
\end{equation}
The ODE \eqref{phi} is autonomous and can be solved explicitly
by separation of variables. 
Indeed, we have
\begin{displaymath}
  h'(\tilde\phi) \tilde\phi' = h^2(\tilde\phi), 
  \quad\rightarrow\quad
  \frac{d}{dq}h\left(\tilde\phi\right)=h^2\left(\tilde\phi\right)
  \quad\rightarrow\quad \frac{dh}{h^{2}}=dq.
\end{displaymath}
Integrating $q$ over $(q,0)$, and $h$ over $(h(\tilde\phi(q)), h\left(1\right))$, we get
\begin{equation}
  \label{tw11}
  \frac{1}{h(\tilde\phi(q))} - \frac{1}{h(1)} = - q, 
  \quad\rightarrow\quad
  h(\tilde\phi(q))=\frac{h(1)}{1-h(1)q}.
\end{equation}
This gives an explicit formula for $\tilde\phi(q)$, i.e,
\begin{equation}
  \label{phif}
  \tilde\phi(q) = h^{-1} \left(\frac{h\left(1\right)}{1-h\left(1\right)q} \right),
\end{equation}
where $h^{-1}$ denotes the inverse mapping of the function $z\mapsto h\left(z\right)$. 
By \eqref{hps} we know that $h'>0$ for $0\le z\le 1$, therefore the inverse is well-defined.
Furthermore, since $\tilde\phi' >0$, the function $\tilde\phi(q)$ is strictly increasing.

We observe that, if $h(0)>0$, the solution   $\tilde\phi(q)$ 
reaches 0 at  $q=-\dhk$ where $\dhk$ is a finite value.
By \eqref{tw11}, we have
\begin{equation}\label{Dhk}
\dhk ~\dot= ~\frac{1}{h(0)} - \frac{1}{h(1)}, \qquad \tilde\phi(-\dhk)=0.
\end{equation}

We now define the function
\begin{equation}  \label{eq:phi}
  \phi(q) ~\dot=~
  \begin{cases}
    \tilde\phi (q)&\text{ for } q\in\left[-\dhk,0\right],\\
    0&\text{ for } q<-\dhk.
  \end{cases}
\end{equation}

We observe that if $h(0)=0$, then $\dhk= +\infty$. 
If the total drop $D$ is finite, we have 
\begin{equation}\label{zetamin}  
\phi(q) \ge c_o, \quad \text{ for }q\in\left]-D,0\right], \quad \text{where}
\quad c_o = \phi(-D) >0\,.
\end{equation}

\begin{remark}\label{rm:00}
If $h(0)=0$, this means we have $g(0)=0$ and $f'(+\infty)=0$. 
It is observed in \cite{AS-se} that the slope 
$w=u_x(t,x)$ remains uniformly bounded for all $t$.
Correspondingly $z(t,u)\ge c_{o}$ for all $(t,u)$ and for some positive constant $c_{o}>0$. 
This implies  
\begin{displaymath}
  \zeta(t,q)\ge c_o > 0, \quad \forall t >0\text{ and }\forall q \in\left[-D,0\right].
\end{displaymath}
We have the same observation in the traveling wave solution \eqref{zetamin}.
\end{remark}

We immediately have the next Lemma which 
provides a lower  bound on the derivative of the smooth stationary profile.

\begin{lemma}\label{Dphi}
   If $h(0)>0$, the smooth stationary profile $\phi$ satisfies
   \begin{equation}\label{eq:Dphi1}  
     \phi'(q) \ge c_1, \qquad c_1 = \frac{h(0)^{2}}{\max_{0<z<1} h'(z)}>0, \qquad -\dhk < q<0\,.
   \end{equation}
   If $h(0)=0$, and let $D$ be the total drop, we have
   \begin{equation}\label{eq:Dphi2}  
     \phi'(q) \ge c_2, \qquad c_2 = \frac{h(\phi(-D))^{2}}{\max_{\phi(-D)<z<1} h'(z)}>0, \qquad -D< q<0\,.
   \end{equation}
\end{lemma}

\begin{example}
  \label{ex:1}
Let us choose the following erosion function
$$ 
f(w) = w - \frac{1}{w}, \qquad
g(z) = 1-z^2, \qquad
h(\zeta) = 1+\zeta,
$$ 
By \eqref{phif}, the smooth profile satisfies
$$ 
\phi(q) = h^{-1} \left(\frac{2}{1-2q}\right) = \frac{2}{1-2q} -1 = \frac{1+2q}{1-2q}, \qquad 
q\le 0.
$$ 
We see that $\phi(-0.5)=0$, so $\dhk=0.5$.
\end{example}

\subsection{Wave speeds, admissible conditions and stationary waves}
To understand how the fronts move in the solution $\zeta(t,q)$, 
we derive here the wave speeds, for various types of singularities.  
Furthermore, we discus their admissibility conditions,
following the Lax's entropy condition.
We also single out the cases where the fronts are stationary. 

To simplify notation, we denote the integral term in the $(t,u)$ coordinate by
\begin{equation}
  \label{eq:defG}
  G\left(z;u\right)=\exp \int_{u}^{+\infty}g\left(z\left(t,v\right)\right)\dv\,, \qquad
  G_u(z;u) = -g(z) \cdot G(z;u)\,.
\end{equation}
Note that if $\zeta(t,q)=z(t,u(q))$ where $q(t,u)$ is the drop 
function defined in \eqref{defq},
we have
$ 
  F(\zeta;q) = G(z;u)
$. 

Let $u(t)$ 
be the location of a discontinuity, in a piecewise smooth solutions $z(t,u)$.
The wave speed $\dot u(t)$  for three types of discontinuities in $z(t,u)$ 
were derived in \cite{CGS}. 
Let $q(t)$ be the location of the corresponding discontinuity 
in the solution $\zeta(t,q)$.   
Formally, by the conservation law \eqref{1.2z} we obtain
\begin{equation}\label{relation}
  \dot q(t)= -\left[z(t,u+)-1\right]\dot u(t) -g\left(z(t,u+)\right)
  G\left(z;u\right).
\end{equation}
Thanks to \eqref{relation}, we can now list the corresponding wave speed $\dot q(t)$.

\paragraph{Kink.} 
A concave kink in $u(t,x)$ corresponds to a downward jump in $z(t,u)$.
Since $-g(z)$ is concave, only downward jumps are admissible. 
We can write
\begin{equation*} z^{-} = z(t,u(t)-), \qquad z^{+} = z(t, u(t)+), \qquad
1 \ge z^{-} > z^{+} >0.
\end{equation*}
The speed of this wave is determined  
by the Rankine-Hugoniot condition for \eqref{1.2z} (see \cite{CGS}),
$$ 
\dot u(t) =- G(z;u) \cdot 
\frac{g\left(z^{+}\right)-g\left(z^{-}\right)}{z^{+}-z^{-}}.
$$ 
By using the relation \eqref{relation},
we have 
\begin{equation*}
\dot q(t)=G\left(z;u\right)
    \left\{\left(z^{+}-1\right)\frac{g\left(z^{+}\right)-g\left(z^{-}\right)}
      {z^{+}-z^{-}}-g\left(z^{+}\right)\right\}\,.
\end{equation*}
Using the functions $h$ and $F$, we can write the speed as
\begin{equation}
  \label{eq:qvk}
      \dot q(t) =-F\left(\zeta;q\right)\left(1-z^{+}\right)\left(1-z^{-}\right)
    \frac{h\left(z^{-}\right)-h\left(z^{+}\right)}{z^{-}-z^{+}}.
\end{equation}
Since $z^{+}<1$, the kink is stationary if and only if  $z^{-}=1$.

\paragraph{Hyper-kink.}
A hyper-kink in $u(t,x)$ corresponds to 
a downward jump in $z(t,u)$ with $z^{+}=0$. 
By taking the limit $z^{+}\to 0$ in \eqref{eq:qvk},
we get 
\begin{equation}\label{speedHK}
\dot q(t)  = -F\left(\zeta;q\right)\left(1-z^{-}\right)
    \frac{h\left(z^{-}\right)-h(0)}{z^{-}}.
\end{equation}
Note that a hyper-kink is admissible only if the jumps in $z$ is downward. 
Furthermore, \eqref{speedHK} indicates that 
it is stationary if and only if $z^{-}=1$.

\paragraph{Shock.} 
A jump in the profile $u(t,x)$ corresponds to
an interval where $z(t,u)=0$. 
Let $z(t,u)=0$ on the interval $u^{-} < u < u^{+}$, and we write 
\begin{equation*}
z^{+}(t) = z(t,u^{+}(t)+) >0, \qquad z^{-}(t) = z(t,u^{-}(t)-) >0, \qquad
\Delta = u^{+} - u^{-}.
\end{equation*}
Here $\Delta$ is the size of the shock. 
In the $(t,q)$ coordinate, we let $ q^{-} < q < q^{+}$ denote the corresponding 
interval of the shock, and $\Delta=u^{+}-u^{-} = q^{+}-q^{-}$.
This shock  gives two fronts, 
namely the left front  $q^{-}(t)$  and the right front $q^{+}(t)$.

We now introduce the function
\begin{equation}\label{xidef} 
  \psi(s) ~\dot=~\frac{e^{h(0)s}-1}{s} \quad (s>0),\qquad \psi(0)=h(0)\ge 0\,.
\end{equation}
The derivative of $\psi$  satisfies
\begin{equation}\label{xi'} 
  \psi'(s) = \frac{h(0)s e^{h(0)s} -e^{h(0)s}+1}{s^2} >0, \qquad (s>0)
\end{equation}

Let's first consider the right front $q^{+}(t)$.
The speed $\dot u^{+}(t)$ could be obtained by the Rankine-Hugoniot condition for 
\eqref{1.1} (see \cite[Subsection 2.1]{CGS} for details)
\begin{displaymath}
  \dot u^{+}(t)=-\frac{G\left(z;u^{+}\right)}{z^{+}}  \left( 
  g\left(z^{+}\right)-\psi(\Delta)\right)\,.
\end{displaymath} 

Using again \eqref{relation}, we get
\begin{equation*}
\dot q^{+}(t) = G\left(z;u^{+}\right)\frac{1-z^{+}}{z^{+}} 
    \left[  \psi(\Delta)     - \frac{g(z^{+})}{1-z^{+}}    \right]\,.
\end{equation*}
Using the functions $h$ and $F$, we get
\begin{equation}  \label{eq:srv}
  \dot q^{+}(t)  = F\left(\zeta;q^{+}\right)\frac{1-z^{+}}{z^{+}}\left[
      \psi(\Delta)       -h\left(z^{+}\right)\right]\,.
\end{equation}

A similar computation gives us the speed for the left front $q^{-}$, 
\begin{equation}  \label{eq:slv}
    \dot q^{-}(t)=F\left(\zeta;q^{-}\right)\frac{1-z^{-}}{z^{-}}\left[
      \psi(\Delta) e^{-h(0)\Delta}      -h\left(z^{-}\right)\right]\,.
\end{equation}

We now discuss the admissible conditions. 
By Lax entropy condition,
characteristics could only enter or stay parallel to shock curves. 
One can easily check that
the left front at $q^{-}$ is always admissible, see \cite{CGS,ShZh}.
For the right front $q^{+}$, by using \eqref{char1}, Lax condition yields the following 
inequality 
\begin{equation}  \label{en-z}
  h\left(z^{+}\right)-z^{+}\left(1-z^{+}\right)h'\left(z^{+}\right)    \le \psi(\Delta)\,. 
\end{equation}

Condition \eqref{en-z} gives an upper bound for the value of $z^{+}$ 
for a given shock size $\Delta$. 
To this end, it is convenient to define a function $\adm\left(\Delta\right)$ 
that gives the maximum value of $z^{+}$ such that the entropy condition 
\eqref{en-z} holds for a shock size of $\Delta$,
i.e., when \eqref{en-z} holds with equal sign. 
Therefore, we define the mapping $\Delta \mapsto \adm$  implicitly by 
\begin{equation}
  \label{eq:defent}
  h\left(\adm\right)-\adm\left(1-\adm\right)h'\left(\adm\right)
    = \psi(\Delta)\,.
\end{equation}
We see that this implicit definition is well-defined. 
Indeed, we know $\psi'>0$.
By the third property in \eqref{hps} we have
\begin{equation}\label{eq:ggg}
 \frac{d}{dz} \left( h(z)-z(1-z)h'(z) \right)=z\left[2h'(z)-(1-z)h''(z)\right]>0, \qquad (0<z\le 1)\,.
\end{equation}

We now check the condition when the fronts are stationary. 
The left front $q^{-}$ is stationary if and only if $z^{-}=1$, according to
\eqref{eq:slv}.
For the right front $q^{+}$,  
we define  a function $\stat\left(\Delta\right)$ such that
for any given shock size $\Delta$, we have  
$\dot q=0$ if $z^{+}=\stat(\Delta)$. 
Thus, by the front speed \eqref{eq:srv},
the mapping $\Delta \mapsto \stat$ is implicitly defined as
\begin{equation}  \label{eq:defzero}
  h\left(\stat\right)= \psi(\Delta)\,. 
\end{equation}
Both $h$ and $\psi$ are strictly increasing functions,
therefore the implicitly function \eqref{eq:defzero} is well-defined.

For a given shock size $\Delta$, by \eqref{eq:srv} we have
\begin{eqnarray}
\dot q^{+} <0 \qquad &\text{if}& \qquad z^{+} > \stat\left(\Delta\right)\,,\label{RS-}\\
\dot q^{+} >0 \qquad &\text{if}& \qquad z^{+} < \stat\left(\Delta\right)\,.\label{RS+}
\end{eqnarray}

{From} \eqref{eq:srv} we also observe that,
if $z^{+}=1$, the right front of the shock is stationary, provided 
that it satisfies the admissible condition \eqref{en-z}.
Let $\dss$ denote the smallest shock size such that $z^{+}=1$ is admissible.
Then $\dss$ satisfies the equation
\begin{equation}  \label{Dss}
  \psi(\dss)=h(1)\,.
\end{equation}
By \eqref{xi'}, the function $\psi$ is strictly increasing, 
therefore the value $\dss$ is uniquely determined by \eqref{Dss}. 

We can now conclude that any shock with $z^{-}=z^{+}=1$ and the 
shock size larger than $\dss$ is stationary.

\bigskip

\begin{example}
   If the erosion function has the simple form $g(z)=1-z^{2}$, as in Example \ref{ex:1},
   then $h(z)=1+z$ and $h(0)=g(0)=1$.  In this case, 
   the implicit functions can be expressed explicitly. We have
   \begin{eqnarray*}
     \adm \left(\Delta\right)~=~ \sqrt{\frac{e^\Delta-1}{\Delta}-1}, \qquad
     \stat\left(\Delta\right) ~=~ \frac{e^\Delta-1}{\Delta} -1,&&
     \frac{e^{\dss}-1}{\dss}~=~2,  \qquad \dss\approx 1.256.
   \end{eqnarray*}
Let $D$ denote the total drop. 
For various values of $D$, we plot the functions 
$\phi(q)$, $\adm(D+q)$,
$\stat(D+q)$ in Figure \ref{fig:03}, 
using red, blue and green colors respectively. 
These graphs give us insights in the construction of the stationary 
traveling waves.
Consider the intersection point of the graphs $\phi$ (in red) 
and $\adm$ (in green), and let $\bar q$ be its $q$ coordinate.
Then, a shock with $z=0$ on $q\in[-D, \bar q]$, with the right front 
value $\adm(D+\bar q)$ would be stationary. 
{From} Figure \ref{fig:03} we see that, 
if $D< \dhk$, the two curves do not intersect, so
no stationary shock could exist. 
If $\dhk<D<\dss$, then there exists only one intersection point, 
indicating one possible stationary shock. 
Finally, if $D>\dss$, the two curves do not intersect.
However, since $D>\dss$, then a shock with $z=0$ on $q\in [-D,0]$
would be admissible with $z^{+}=1$, therefore it is stationary.
\end{example}

\begin{figure}[htb]
  \centering
  \includegraphics[width=14cm]{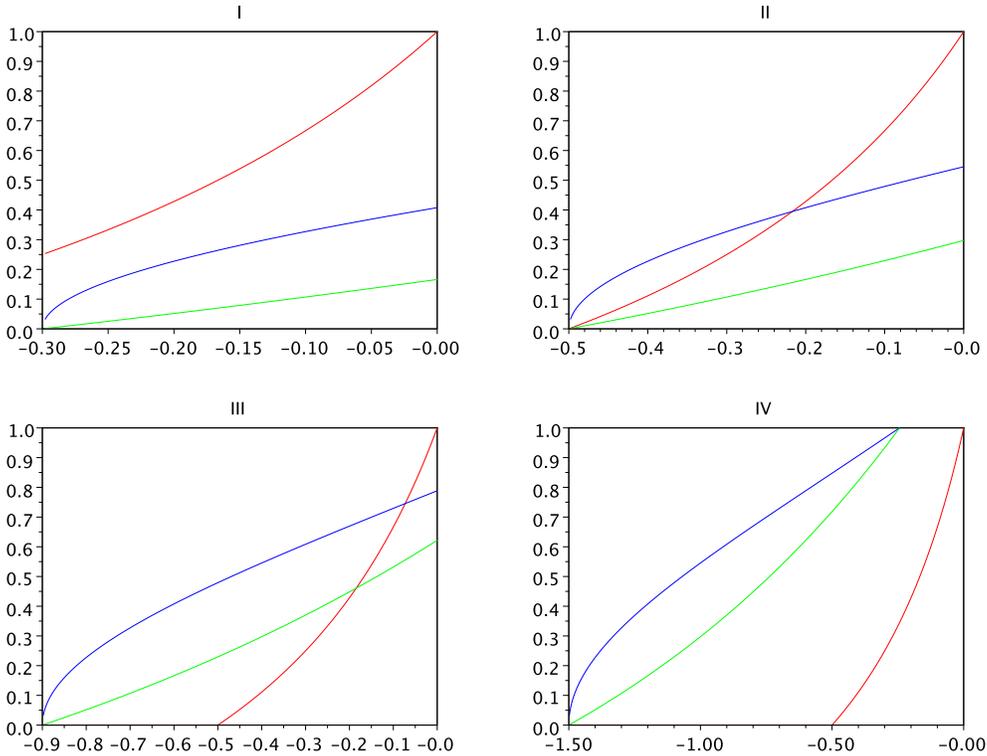}
  \caption{The graphs of $\phi(q)$ (in red), 
    $\adm(D+q)$ (in blue) and  $\stat(D+q)$ (in green) for $4$
    different values of $D$. (I): $D<\dhk$. (II): $D=\dhk$.
    (III): $\dhk<D<\dss$. (IV): $D>\dss$.}
  \label{fig:03}
\end{figure}

\subsection{Technical Lemmas} 
Inspired by the graphs in Figure \ref{fig:03}, 
we now prove some technical lemmas.

\begin{lemma}  \label{lm:ZZ2}
  If $h(0)>0$, the followings hold.
  \begin{enumerate}[a)]
  \item
    The functions $\stat\left(\Delta\right)$ and  $\adm\left(\Delta\right)$
    are  strictly increasing for $\Delta\ge 0$;
  \item
    $\stat\left(0\right)=\adm\left(0\right)=0$ and
    $\stat\left(\dss\right)=\adm\left(\dss\right)=1$;
  \item
    $\stat\left(\Delta\right) < \adm\left(\Delta\right)$ for
    any $\Delta\in\left(0,\dss\right)$.
  \end{enumerate}
\end{lemma}

\begin{proof} These properties follow {from} the definition of 
the functions $\stat$ and $\adm$,
and the properties of the functions $h$ and $\psi$. 
\begin{enumerate}[a)] 
  \item 
  Since $h'>0$ and $\psi'>0$, by the definition of $\stat$ in \eqref{eq:defzero} we 
  have $\stat' >0$.  For $\adm$ defined in \eqref{eq:defent}, by 
  the property \eqref{eq:ggg} and $\psi'>0$, we conclude that $\adm'>0$.
  \item 
    When $\Delta=0$, we have $\psi(0)=h(0)$. Since
    $h(0)-0\left(1-0\right)h'(0)=h(0)$, from \eqref{eq:defent} we must have
    $\stat(0)=\adm(0)=0$.
  
  When $\Delta=\dss$, by \eqref{Dss} and \eqref{eq:defzero}, 
  we have $h(\stat) = \psi(\dss)=h(1)$, therefore $\stat(\dss)=1$.
  For $\adm$, by \eqref{eq:defent} we have $h(\adm)=\psi(\dss)=h(1)$,
  therefore $\adm(\dss)=1$. 
  \item
  By the definitions \eqref{eq:defent} and \eqref{eq:defzero}, and the
  property $z(1-z)h'(z)>0$ for $0<z<1$, we immediately conclude $c)$.
\end{enumerate}
\end{proof}

\begin{lemma}\label{lm:ZZ1}
  If $\stat\left(\Delta\right) = \phi(q)$ for some values of 
  $\Delta\ge 0$ and $q\in[-\dhk,0]$ then  
  $$
  \phi'(q)-\stat'(\Delta) \ge \kappa>0\,,
  $$
  for some constant $\kappa>0$ independent of $\Delta$ or $q$. 
  This implies that any horizontal shift of the graph of $\stat$ can
  intersect the graph of $\phi$ at most at one point. 
\end{lemma}

\begin{proof}
If $h(0)=0$, then \eqref{eq:defzero} implies $h(\stat)=h(0)$, which gives
$\stat\left(\Delta\right) \equiv 0$,  so  the Lemma is trivial, and one can 
let 
\begin{equation*}
\kappa=\min_{0<z<1} \frac{h^2(z)}{h'(z)}.
\end{equation*}
We now consider the case $h(0)>0$. 
Let $\Delta$ and $q$ be the values such that  $z=\stat(\Delta)= \phi(q)$.
By \eqref{phi} and \eqref{eq:defzero}, we have 
\begin{equation*}
 h'(\stat) \phi'(q) = h^2(\stat) = \psi^2(\Delta),
\qquad 
h'(\stat) \stat'(\Delta) = \psi'\left(\Delta\right).
\end{equation*}
We have 
\begin{eqnarray*}
&& \phi'(q) - \stat'(\Delta) ~=~ 
\frac{1}{h'(z)} \left[ \psi^2(\Delta)-\psi'(\Delta)\right]
~\ge ~ \frac{1}{\max_{0<z<1} h'(z)}\cdot \min_{s\ge 0}  \left[ \psi^2(s)-\psi'(s)\right]\\
&=& \frac{1}{\max_{0<z<1} h'(z)} \cdot \min_{s\ge 0}  
\frac{e^{h(0)s}\left(e^{h(0)s}-h(0)s-1\right)}{s^2}\\
&=& \frac{1}{\max_{0<z<1} h'(z)} \cdot \lim_{s\to 0+}  
\frac{e^{h(0)s}\left(e^{h(0)s}-h(0)s-1\right)}{s^2}
~=~ \frac{1}{\max_{0<z<1} h'(z)} \cdot \frac12 h^2(0) >0\,.
\end{eqnarray*}
We can now let 
\begin{equation*} 
   \kappa = \frac{h^2(0)}{2 \cdot\max_{0<z<1} h'(z)}   >0\,,
\end{equation*}
completing the proof.
\end{proof}

\begin{lemma}\label{lm:ZZ5}
  If $h(0)>0$, we  have    $\dhk<\dss$.
\end{lemma}

\begin{proof}
  By $b)$ in Lemma \ref{lm:ZZ2} and \eqref{phi}, \eqref{Dhk}, we have 
  \begin{equation*}
    \dss = \int_0^1 \left[\stat^{-1}\right]'(\xi)\, d\xi, \qquad
    \dhk = \int_0^1 \left[\phi^{-1}\right]'(\xi)\, d \xi.
  \end{equation*}
  By Lemma \ref{lm:ZZ1}, $\left[\stat^{-1}\right]'(\xi)>
  \left[\phi^{-1}\right]'(\xi)$, 
  therefore $\dhk < \dss$, 
  completing the proof.
\end{proof}

We immediately have the next Corollary regarding the intersection point
of the graphs of $\phi(q)$ and $\stat(q+D)$.
This will be useful in the construction  of the 
stationary traveling wave profiles and in the 
proof of their uniqueness w.r.t.~total drop.

\begin{corollary}\label{cor:ZZ3}
  Let $D$ be the total drop. We have the following results concerning the intersection points
  of the graph of $\phi(q)$ and $\stat(q+D)$ on the interval $-D\le q\le 0$.
  \begin{itemize}
  \item[(1).] If $D < \dhk$, the two graphs never intersect, and $\stat(q+D)<\phi(q)$
  for $-D\le q\le 0$;
  \item[(2).] If $D = \dhk$, the two graphs intersection once, at $q=-D$, 
  and $\stat(q+D)<\phi(q)$ for $-D< q \le 0$;
  \item[(3).] If $\dhk <D<\dss$, the two graphs intersect once, 
              at $q^{+}$ where $-\dhk<q^{+}<0$, 
              $\stat(q+D) > \phi(q)$ for $q<q^{+}$ and 
              $\stat(q+D) < \phi(q)$ for $q>q^{+}$;
  \item[(4).] If $D = \dss$, the two graphs intersect once at $q=0$, and 
              $\stat(q+D) > \phi(q)$ for $-D < q<0$;
    \item[(5).] If $D > \dss$, the two graphs never intersect, and 
    $\stat(q+D) > \phi(q)$ for $-D<q\le 0$;
  \end{itemize}
\end{corollary}

This Corollary is illustrated in Fig.~\ref{fig:03}.

\section{Existence and uniqueness of traveling waves}

\subsection{Stationary traveling waves for $\zeta(t,q)$}

In this subsection we prove the following Theorem on existence of 
traveling waves and their uniqueness with respect to total drop.

\begin{theorem}\label{thm:tw}
For every  value of the total drop $D$, there exists exactly only one  stationary 
traveling wave profile $\zeta(t,q)=Z(q)$, defined on $q\in[-D,0]$.
\end{theorem}

\begin{proof}
We prove the existence by construction. 
Let $\zeta(t,q)=Z(q)$ be defined on the interval  $q\in [-D,0]$, with 
\begin{displaymath}
  Z(-D)=1, \quad Z(0)=1, \qquad 0 \le  Z(q) < 1, \quad q\in ]-D,0[,
\end{displaymath}
for various values of the total drop $D$.
We recall the values $\dhk$ and $\dss$, defined in \eqref{Dhk} 
and \eqref{Dss}, respectively.
The stationary profile is constructed in different ways for different values of 
total drop $D$.

\textbf{Type 1.} This is the case when $D < \dhk$.  We let 
$$
Z(q)  ~\dot= ~ \begin{cases}
1 & q=-D, \\
\phi(q) \qquad &  -D < q \le 0. 
\end{cases}
$$
Here we have a kink at $q=-D$ where $Z$ has a downward jump,
then it is connected to a smooth stationary profile on the right. 
By \eqref{eq:qvk}, the kink at $q=-D$ is also stationary.

We remark that if $h(0)=0$, the solution $\zeta(t,q)$ remains
uniformly bounded away {from} 0 for all $t>0$.
In this case, only Type 1 profiles are possible. 
For the next 3 types, we assume $h(0)>0$. 

\textbf{Type 2.} In this case we have $D=\dhk$.  We let 
$$
Z(q)  ~\dot= ~ \begin{cases}
1 & q=-D, \\
\phi(q)\qquad  &  -D < q \le 0.
\end{cases}
$$
Here we have a hyper-kink at $q=-D$ and it is connected to a smooth stationary profile
on the right. 
By \eqref{eq:qvk}, the hyper-kink at $q=-D$ is stationary.

\textbf{Type 3.} We consider the case with $\dhk < D < \dss$. 
Let $(q^{+},z^{+})$ be the intersection point of the graphs 
$\phi(q)$ and $\stat(-D+q)$. According to Corollary \ref{cor:ZZ3}
the two graphs intersect only once, at some interior point where $-D < q^{+} <0$. 
At this intersection point we have  
$$
 z^{+} = \phi(q^{+})=  \stat\left(\Delta\right), 
\qquad \mbox{where} \qquad \Delta=q^{+}+D\,. 
$$
We define the profile
$$
Z(q)  ~\dot= ~ \begin{cases}
1 & q=-D,\\
0 &  -D < q \le q^{+}, \\
\phi(q) \qquad &  q^{+} < q \le 0. 
\end{cases}
$$
Here we have a shock on $q\in(-D,q^{+})$, and it is connected to
a smooth stationary profile on the right. 
The left front of the shock, located at $q=-D$, is stationary by \eqref{eq:slv}.  
The  right front at $q^{+}$ is also stationary by construction.

\textbf{Type 4.} This is the last case when $D \ge \dss$.  We let
$$
Z(q)  ~\dot= ~ \begin{cases}
1 & q=-D,\\
0 &  -D < q <0,\\
1 \qquad &  q= 0.
\end{cases}
$$
This is a simple shock which is admissible. 
By \eqref{eq:slv} and \eqref{eq:srv},
both left and right fronts are stationary. 
Therefore, the whole profile is stationary. 

Finally we note that the right front of a shock would be stationary if and 
only if it is an intersection
point of the graph $\phi$ with some horizontal shift of the graph $\stat$.
By Corollary \ref{cor:ZZ3}, there exist at most one such intersection point.
This shows the uniqueness of these profiles with respect to the total drop,
completing the proof.
\end{proof}

\begin{example} \label{ex:2} 
If we choose the erosion function as in Example~\ref{ex:1}, 
then $\dss$ satisfies
\begin{equation*}
2=\frac{\exp\{\dss\}-1}{\dss}, \qquad \dss \approx 1.26.
\end{equation*} 
We have $\dhk=0.5$ {from} Example~\ref{ex:1}.
The plots of these 4 types of traveling waves 
are given in Figure~\ref{fig:zq4}.
\end{example}

\begin{figure}[ht]
\begin{center}
\includegraphics[width=14cm]{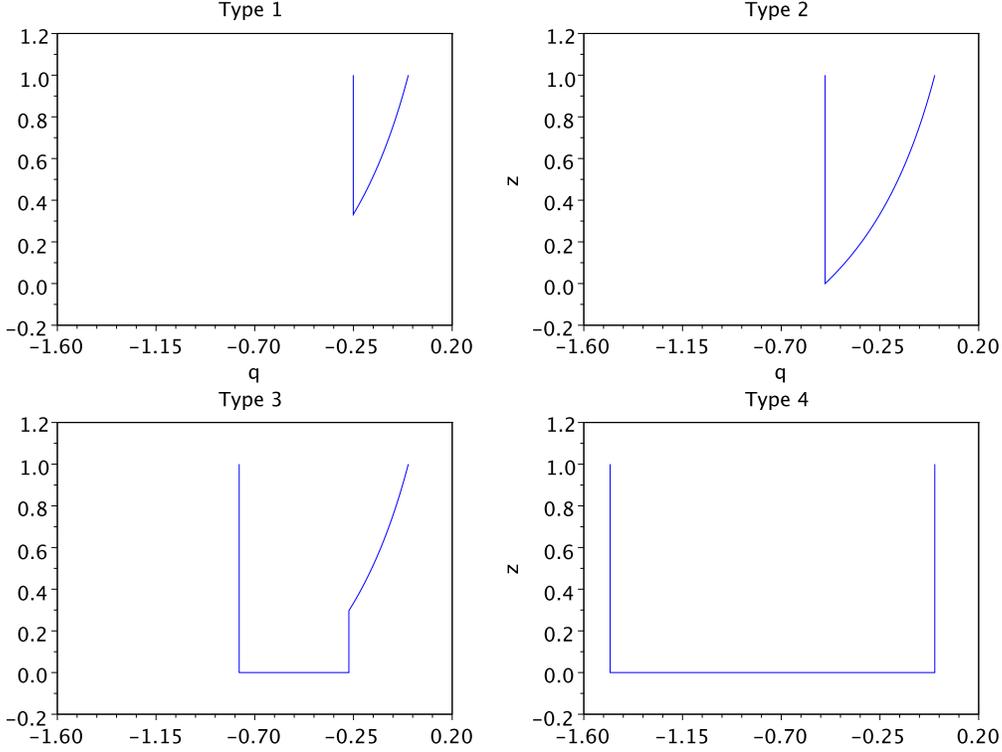}
\caption{Stationary traveling wave profiles $\zeta(t,q)=Z(q)$.}
\label{fig:zq4}
\end{center}
\end{figure}

\subsection{Moving traveling waves for $u(t,x)$}

For the original physical variables $u(t,x)$ and the slope 
$w(t,x)=u_x(t,x)$, the traveling waves are not stationary. 
They share some interesting properties which we would like to comment. 
Recall \eqref{2.1}, and let $W(\xi)$  be such a traveling wave,
and let $\mathcal{F}(\xi)$ be the corresponding integral term.
{From} the equation \eqref{1.2} we get
$$
-\sigma W'(\xi) + (f(W(\xi)) \, \mathcal{F}(\xi))' =0.
$$
Integrating it once, 
by the boundary condition  \eqref{2.1}
we get
\begin{equation}\label{ode0}
  -\sigma W(\xi) + f(W(\xi)) \mathcal{F}(\xi) = -\sigma,
  \qquad
  \sigma = \frac{f(W(\xi))}{W(\xi)-1} \cdot \mathcal{F}(\xi)
\end{equation}
This gives us  the constant propagation speed 
\begin{equation}\label{sigma}
  \sigma ~= ~ \frac{f(W(\xi))}{W(\xi)-1} \cdot \mathcal{F}(\xi) =
  \lim_{\bar\xi\to+\infty} \frac{f(W(\bar\xi))}{W(\bar\xi)-1} \cdot
  \mathcal{F}(\bar\xi)= 
  f'(1)\,, \qquad \forall \xi\,.
\end{equation}
This says that, for any traveling wave that consists of a smooth part,
(i.e., Type 1, 2 and 3), 
they must all travel with speed $\sigma=f'(1)$!

We now derive the ODE satisfied by $W(\xi)$. 
Taking logarithm function on both sides of \eqref{ode0}, we get
\begin{equation*} \ln \sigma + \ln (W-1) = \ln f(W) + \int_\xi^\infty f(W(y))\dy\,.\end{equation*}
We now differentiate both sides of this equation with respect to $\xi$, and we get
an autonomous ODE with a boundary condition
\begin{equation}\label{Wode}
W'(\xi) = - \frac{f^2(W) (W-1)}{f(W)-f'(W)(W-1)},\qquad
\lim_{\xi\to+\infty} W(\xi) = 1.
\end{equation}
Integrating $W$ once, we obtain the corresponding profile for $u(t,x)$.

If the shock size $\Delta$ is bigger than $\dss$, we have Type 4, 
which is a single shocks connected to $u_x=1$ on
both the left and the right sides.
It travels with the shock speed (see \cite[(2.11)]{ShZh}):
\begin{displaymath}
 \sigma = \frac{e^{\Delta\cdot f'(+\infty)}-1}{\Delta}. 
\end{displaymath}

As a consequence of Theorem \ref{thm:tw}, we conclude that such
traveling waves exist for $u(t,x)$ and $w(t,x)$, 
and they are unique w.r.t.~the total drop.

\begin{example}\label{ex:W}
Consider the erosion function $f(w)=w-\frac{1}{w}$,
same as in Example \ref{ex:1},
then \eqref{Wode} becomes
$$
W'(\xi)= -(W+1)^2(W-1).
$$
Integrating $W(\xi)$ in space, one obtains the traveling waves for the 
profile height $u(t,x)$. 
These are plotted in Figure~\ref{fig:3T} for Type 1, 2 and 3, at $t=1$. 
All these waves travel with the same speed $\sigma=f'(1)=2$.
\end{example}

\begin{figure}[htbp]
\begin{center}
\includegraphics[width=14cm]{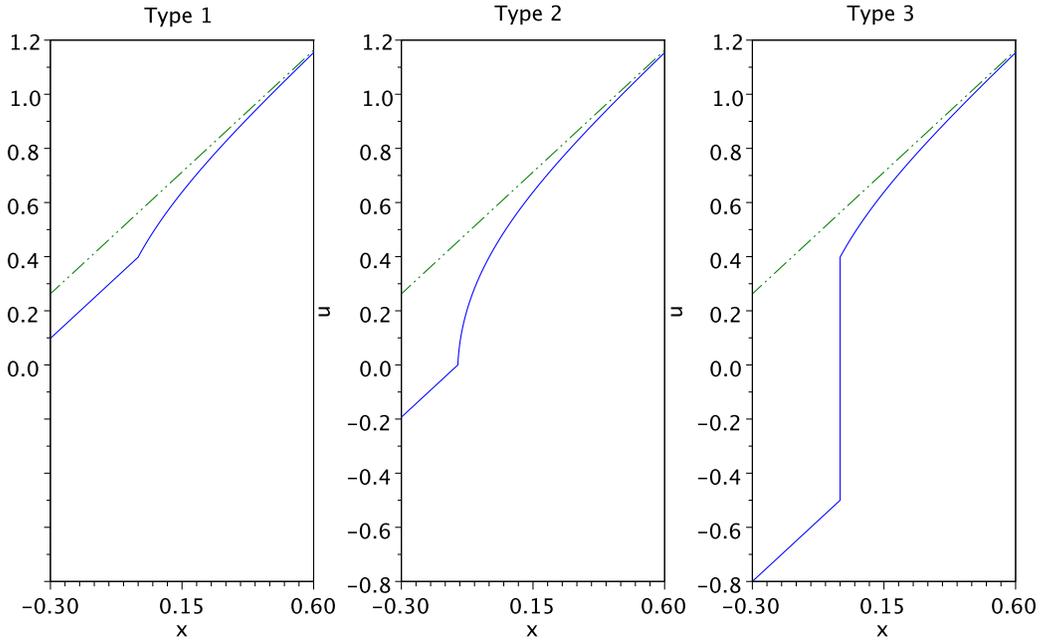}
\caption{3 Types of traveling wave profiles in the physical coordinate $u(1,x)$.}
\label{fig:3T}
\end{center}
\end{figure}

\section{Stability of traveling waves}

In this Section we prove the local stability for traveling waves.

\begin{theorem}\label{thm:stab}
Let $z(t,u)$ be the solution of the Cauchy problem for \eqref{1.3} with initial data
satisfying \eqref{cond-z} and \eqref{z02}, and let $\zeta(t,q)$ be defined in 
\eqref{eq:defs}. Let $D$ be the total drop, and $Z(q)$ be the corresponding 
stationary traveling wave profile. 
Then, for every $\varepsilon>0$, there exists a finite 
value $T^\varepsilon$ such that 
\begin{equation}\label{eq:stab}
  \left\| \zeta(t,\cdot)-Z(\cdot)\right\|_{\mathbf{L}^1} \le C
  \varepsilon\,,
  \quad \text{ for all }\quad t\ge T^{\varepsilon}.
\end{equation}
Here the constant $C$ is independent of $\varepsilon$.
\end{theorem}

The rest of this Section is devoted to the proof of Theorem \ref{thm:stab}.

\subsection{Structure of the solution $\zeta(t,q)$ for non-decreasing data} 

Thanks to the additional assumption \eqref{z02}, 
we now make the following assumptions on the solutions $\zeta(t,q)$.

\begin{assumption}\label{prop:CGS}
Let $z(t,u)$ be the solution of the Cauchy problem for \eqref{1.3} with initial data
satisfying \eqref{cond-z} and \eqref{z02}.
Let $\zeta(t,q)$ be defined in \eqref{eq:defs}. It satisfies the following properties.
   \begin{itemize}
    \item There is a stationary downward jump at $q=-D$, with $\zeta(t,-D-)=1$, for all $t\ge0$;
    \item There is maximum one shock in the solution, with the left front at $q=-D$,
    and the right front at $q^{+}(t)$ where $-D\le q^{+}(t)\le 0$.
    \item For any $t>0$, $\zeta(t,q)$ is locally Lipschitz continuous and strictly  increasing 
    on the interval $\tilde q(t) < q \le 0$, where $\tilde q(t)=q^{+}(t)$ 
    if there is a shock, and $\tilde q(t)=-D$ if no shock exists;
    \item The total variation of $q\mapsto\zeta(t,q)$ is uniformly bounded by $2$ for all $t\ge0$.
   \end{itemize}
\end{assumption}

\begin{remark} 
The properties in Assumption \ref{prop:CGS} are expected in the solution $\zeta(t,q)$.
A rigorous proof could be carried out through front tracking piecewise constant 
approximate solutions. 
However, such a proof could be lengthy, since one has to enter the details
of setting up the front tracking algorithm, 
then establish the a-priori estimates and the convergence. 
Since this current paper is focused on traveling waves and their properties, 
we would not provide the detailed proof and state these as assumptions.
Below we provide some formal arguments to support these assumptions.

\textbf{(1).} 
The assumptions \eqref{cond-z} and \eqref{z02} imply the following properties
for $\zeta(0,q)$,
$$
  \zeta(0,-D)=1,\quad \zeta(0,q_1)<\zeta(0,q_2) \quad (q_1<q_2), \quad 
  \zeta(0,0)=1\,. 
$$
For smooth solutions $\zeta(t,q)$, the characteristic equations 
\eqref{char1}-\eqref{char2} hold.
Therefore $\zeta$ decreases along characteristics. 
The assumption \eqref{z02} implies that  $\zeta(t,q)<1$ for all $t>0$ and $-D<q<0$. 

\textbf{(2).} 
The integral term $F(\zeta;q)$ is strictly decreasing in $q$, i.e,
$$
  F(\zeta;q_1) > F(\zeta; q_2), \qquad  (-D<q_1<q_2<0)\,.
$$

\textbf{(3).} 
We consider two points in the solutions along the characteristics,
$\zeta(t,q_1(t))$ and $\zeta(t,q_2(t))$, with $-D<q_1(0)<q_2(0)<0$
and $\zeta(0,q_1(0)) \le \zeta(0,q_2(0))$.
As long as $\zeta(t,q_1(t))>0$ and $\zeta(t,q_2(t))>0$,
the equations \eqref{char1}-\eqref{char2} hold. Let $\bar t$ be a
first time such that
\begin{displaymath}
  \zeta(\bar t,q_1(\bar t)) = \zeta(\bar
  t,q_2(\bar t))\quad\text{ or }\quad q_1(\bar t)=q_2(\bar t).
\end{displaymath}
Since the maps $z\mapsto \left(1-z\right)^{2}h'(z)$, $q\mapsto
F\left(\zeta;q\right)$ are decreasing and positive, we have
\begin{displaymath}
  \frac{d}{dt}\left[q_{2}(t)-q_{1}(t)\right]=\dot q_{2}(t)- \dot
  q_{1}(t)\ge 0.
\end{displaymath}
This implies that $q_{2}(\bar t)=q_{1}(\bar t)$ cannot happen.
If $\zeta(\bar t,q_1(\bar t)) = \zeta(\bar
t,q_2(\bar t))$ holds, we can compute
\begin{displaymath}
  \frac{d}{dt}\left[\zeta(\bar t,q_{2}(\bar t)) - \zeta(\bar
  t,q_{1}(\bar t))\right]= g\left(\zeta\left(\bar t,q_{1}(\bar
      t\right)\right)\left[F\left(\zeta ; q_{1}(t)\right)-F\left(\zeta ; q_{2}(t)\right)\right]>0
\end{displaymath}
which gives a contradiction.
In other words, for any $t>0$, if $\zeta(t,q)\ge0$,
then $q\mapsto \zeta(t,q)$ is non--decreasing and the rarefaction fan is strictly spreading.

\textbf{(4).}
The total variation of $\zeta(t,\cdot)$ is uniformly bounded by $2$ for $t\ge 0$.

\textbf{(5).} If $\zeta(t,q)$ is strictly increasing, then
there must be a downward jump in $\zeta(t,q)$ at $q=-D$, i.e.,
\begin{equation*} \zeta(t,-D)=1, \quad \zeta(t,-D+) <1, \qquad (t>0)\,.\end{equation*}
This jump is stationary.

\textbf{(6).}
If $h(0)>0$, shocks might form in the solution.  
Since $\zeta(t,q)$ is strictly increasing, there could be maximum one shock, 
with the left front at $q=-D$ and the right front at some
$q^{+}$ where $-D\le q^{+}(t)\le 0$ for all $t\ge0$.

\textbf{(7).}
The solution $\zeta(t,q)$ remains smooth and strictly increasing 
on the part where $\zeta(t,q)>0$. If there are no shocks, this is valid on the
whole interval $-D < q \le 0$. If there is a shock, 
then the interval is $q^{+}(t) < q \le 0$ where $q^{+}(t)$ is the position 
of the right front of the shock.
\end{remark}

\subsection{Properties of rarefaction fronts; formal arguments.}

Thanks to Assumption \ref{prop:CGS}, we now only need to 
trace the location of the right front $q^{+}(t)$  of the possible shock, 
and study the evolution of the rarefaction fronts by characteristic 
equations \eqref{char1}-\eqref{char2}.
Next Lemma follows {from} \eqref{char1}-\eqref{char2}.

\begin{lemma}\label{lm:4a}
  Let $\zeta(t,q(t))>0$ be a point on the rarefaction fan. 
  Then, as $t$ grows, the trajectory $t\mapsto (q,\zeta)$
  matches some horizontal shift of the graph of $\phi$.
  Furthermore, the point $(q(t), \zeta(t,q(t)))$ travels to the left
  and downwards, until it merges into a singularity, at some $t\le T_{\zeta_{o}}$ 
  where 
  $T_{\zeta_{o}}=\frac{D}{\left(1-\zeta_{o}\right)\min_{0\le z \le 1}h'(z)},\qquad
  \zeta_{o}=\zeta\left(0,q(0)\right)<1$.
\end{lemma}

\begin{proof}
It suffices to observe that
$$
  \phi'(q)=\frac{h^2(\phi)}{h'(\phi)},\quad 
  \frac{\dot \zeta(t,q(t))}{\dot q(t)} = \frac{h^2(\zeta)}{h'(\zeta)},\quad
  \dot q(t) <0, \quad \dot \zeta(t,q(t))<0\,,
$$
therefore setting $q_{o}=q(0)<0$ and
$\zeta_{o}=\zeta\left(0,q(0)\right)<1$, we have
\begin{displaymath}
    \dot
    q(t)=-\left(1-\zeta\left(t,q(t)\right)\right)h'\left(\zeta\left(t,q(t)\right)\right)
    F\left(\zeta;q(t)\right)\le -\left(1-\zeta_{o}\right)\min_{0\le z
      \le 1}h'(z)<0, 
\end{displaymath}
and the curve $q(t)$ has to merge into a singularity before the time
$T_{\zeta_{o}}$.
\end{proof}

\bigskip

The formal arguments for asymptotic analysis is now rather simple, 
thanks to Lemma \ref{lm:4a}. We have the following observations.
\begin{itemize}
\item 
Lemma \ref{lm:4a} implies that, as $t\mapsto +\infty$, 
the remaining rarefaction fan in the solution 
is generated near the point $(q,z)=(0,1)$.
Again, since all rarefaction fronts travel along some horizontal shifts of $\phi(q)$,
we see that the smooth part of the solution $\zeta(t,q)$
must approach the graph of $\phi(q)$ as $t\to+\infty$.
\item
  If a shock forms in the solution, by Corollary \ref{cor:ZZ3}
  and \eqref{RS-}-\eqref{RS+}, it must settle at the corresponding 
  location of the stationary shock front in $Z(q)$.
\end{itemize}
  Combining these two observations, one may conclude that $\zeta(t,q)$
  converges to $Z(q)$ in $\mathbf{L}^1$ as $t\to+\infty$.
  
However, there is a complication. 
It is observed in \cite{CGS,ShZh} that, 
because of the admissible condition \eqref{en-z},
characteristic curves could come out of
the right front of the shock in a tangent direction at some $t>0$. 
Therefore, the formal argument alone is not sufficient.
Instead, in our proof we will construct suitable upper and lower envelopes 
for the solution $\zeta(t,q)$.

\subsection{Upper envelopes.}

The upper envelopes might take 2 stages, depending on the type of stationary profiles.
In stage 1 we control the smooth part of the solution.  
We show that the rarefaction fan gets very close to the stationary traveling
wave after sufficiently long time. 


\begin{lemma}\label{lm:UB}
Let $D$ be the total drop, and
let $\varepsilon>0$ be given.  We define the function
$$
\phi^{+}(q)~\dot=~ \begin{cases}
   1 & q=-D,\\
   \phi(q+\varepsilon) \qquad &-D < q \le -\varepsilon,\\
   1  & -\varepsilon < q \le 0.
 \end{cases}
$$
Then there exists a  time $ \Ta$, 
such that for $t\ge  \Ta$, we have
\begin{equation*} \zeta(t,q) \le \phi^{+}(q), \qquad -D<q\le 0\,.\end{equation*}
If the initial data satisfies 
$\zeta(0,q)\le\phi(q)$ for $-\varepsilon<q<0$, 
we can simply take $\phi^{+}(q)=\phi(q)$.
\end{lemma}

\begin{proof}
  Let $\varepsilon>0$ be given and define $q_{o}=-\varepsilon<0$,
  $\zeta_{o}=\zeta\left(0,q_{o}\right)<1$. 
  We consider the rarefaction fronts generated on the interval $\left(q_{o},0\right)$. 
  By Lemma \ref{lm:4a} they travel along horizontal shifts of the graph
  of $\phi^{+}$, therefore they stay below this graph.
  Let $t\mapsto q$ denote the characteristic curve with $q(0)=q_{o}$. 
By Lemma \ref{lm:4a}, the point $(q(t),\zeta(t,q(t)))$
will merge into a singularity before the time $T_{1}^{\varepsilon}=T_{\zeta_{o}}$.
After that, the smooth part of the solution is the rarefaction fan generated 
on the interval $\left(q_{o},0\right)$ at $t=0$. 
Therefore, we have $\zeta(t,q) \le \phi^{+}(q)$ for $t\ge T_{1}^{\varepsilon}$. 
 
Finally, if $\zeta(0,q)\le\phi(q)$ for $-\varepsilon<q<0$, 
we can simply take $\phi^{+}(q)=\phi(q)$ and carry out the whole argument in a complete similar
way,  completing the proof.
\end{proof}

\bigskip

If the stationary profile is of Type 1 and Type 2, the upper envelopes are complete.
We now consider Type 3 and Type 4, and enter Stage 2 to control the location of the 
right front of the shock.

\begin{lemma}\label{lm:UB2}
Assume $D>\dhk$, and let $q^{+}\le 0$ be the location of the right front
of the shock in the stationary traveling wave profile $Z(q)$. 
Let $\varepsilon>0$ be given. 
There exists a  time $ \Tb$, such that
$$
 \zeta(t,q) \le Z^{+}(q), \qquad -D<q\le 0\,, \qquad t\ge  \Ta+\Tb\,,
$$
where the function $Z^{+}$ is defined as
\begin{equation}\label{Z+}
  Z^{+}(q) ~\dot=~ \begin{cases}
  1, \qquad & q=-D,\\
  0, \qquad & -D<q\le \hat q, \qquad \hat q = q^{+} - C_1 \varepsilon\,,\\ 
  \phi^{+}(q), \qquad &\hat q < q \le 0\,.
  \end{cases} 
\end{equation}
Here the constant $C_1$ does not dependent on $\varepsilon$.
\end{lemma}

\begin{proof}
Let $D>\dhk$ be the total drop, and let 
$\bar q$ be the intersection point of the graphs of
$\phi^{+}(q)$ and $\stat(q+D)$, such that
$\phi^{+}(\bar q) = \stat(\bar q+D)$.
By Lemma \ref{lm:ZZ1}, the functions $\phi$ and $\stat$ are
strictly increasing and transversal. We have
\begin{equation*}
  q^{+}-\bar q \le \bar C \varepsilon, \qquad 
\bar C =  \frac{\max_{q}\phi'( q)}{\kappa}=\frac{1}{\kappa}\max_{0<z<1}\frac{h^2(z)}{h'(z)}\,,
\end{equation*}
where $\kappa$ is defined in Lemma \ref{lm:ZZ1}. 
We can choose the constant in \eqref{Z+} as $C_1=2 \bar C$.

We now construct the upper envelopes, for $t\ge \Ta$, 
$$
  \mathcal{Z}^{+}(t,q) ~\dot=~ \begin{cases}
  1, \qquad & q=-D,\\
  0, \qquad & -D<q\le \tilde q(t) ,\\
  \phi^{+}(q), \qquad &\tilde q(t) < q \le 0,
  \end{cases}
$$
Here the front $\tilde q(t)$ travels with the speed as if it were 
the right front of a shock. By \eqref{eq:srv}, we have
\begin{equation}\label{tqODE}
  \tilde q'(t) = F(\zeta;\tilde q) \cdot \frac{1-\phi^{+}(\tilde q)}{\phi^{+}(\tilde q)}
  \left[\psi(\tilde q+D)-h(\phi^{+}(\tilde q))\right], 
  \qquad \tilde q(\Ta) = -\dhk-\varepsilon\,.
\end{equation}
The ODE \eqref{tqODE} is solved for $t\ge \Ta$, 
until at some time $t=\Ta+\Tb$ when
the front $\tilde q(t)$
passes the one in $Z^{+}$, i.e., when 
$\tilde q(\Ta+\Tb) \ge q^{+}-C_{1}\varepsilon$.

We now show that $\Tb$ is finite for any given $\varepsilon$.
When $\tilde q(t)< \bar q$ the graph of $\phi^{+}$ lies 
strictly below the graph of $\stat(q+D)$. 
Since $q^{+}-C_{1}\varepsilon < \bar q$, therefore by continuity 
we have
\begin{displaymath}
  v_{\varepsilon}=\min\left\{\stat\left(q+D\right)-\phi^{+}(q):\quad
    q\le q^{+}-C_{1}\varepsilon\right\}>0.
\end{displaymath}
Hence, as long as $\tilde q(t)\le q^{+}-C_{1}\varepsilon$, we can compute 
\[
  \begin{split}
    \tilde q'(t) &\ge \left[1-\phi^{+}(\bar q)\right] \left[\psi(\tilde
      q+D)-h(\phi^{+}(\tilde q))\right]\\
    &\ge \left[1-\phi( q^{+})\right]\left[h\left(\stat\left(\tilde q +
          D\right)\right)-h(\phi^{+}(\tilde q))\right]\\
    &\ge \left[1-\phi( q^{+})\right]\cdot\min h' \cdot v_{\varepsilon}.
  \end{split}
\]
We can now conclude that 
$$
\mathcal{Z}^{+} (t,q) \le Z^{+}(q), \qquad \text{for}\quad t\ge \Ta+\Tb\,,
\quad \text{where}\quad 
\Tb= \frac{\dhk + \varepsilon}{\left[1-\phi( q^{+})\right]\cdot\min h' \cdot v_{\varepsilon}}\,.
$$

It remains to show that the solution  $\zeta(t,q)$ satisfies
\begin{equation}\label{comp}
\zeta(t,q)\le \mathcal{Z}^{+}(t,q)\qquad
\text{for} \quad  \Ta \le t \le \Ta+ \Tb\,.
\end{equation} 
Indeed, by construction  \eqref{comp} holds for $t=\Ta$ because
\begin{equation*}
\mathcal{Z}^{+}(\Ta,q)=\phi^{+}(q) \ge \zeta(\Ta,q)\,,
\end{equation*}
where $\phi^{+}$ is the the function defined in Lemma \ref{lm:UB}.  
Now we consider a later time $\bar t\ge \Ta$.
By Lemma \ref{lm:4a}, the smooth part of the solution $\zeta(\bar t,q)$
remains below the graph of $\phi^{+}(q)$. 
It is then enough to check the speed of right front of the possible shock.
Assume that $\zeta(\bar t,q)$ has a shock, with its right front at $\check{q}(t)$, and
\begin{equation*}
\check q(\bar t)=\tilde q(\bar t), \qquad 
\zeta(\bar t, \check q(\bar t)+) \le \phi^{+}(\tilde q(\bar t)).
\end{equation*}
By \eqref{eq:srv} we clearly have 
$
 \check q'(\bar t) \ge \tilde q'(\bar t)
$,
so the graph of $\zeta(t,q)$ remains below that of $\mathcal{Z}^{+}(t,q)$, 
completing the proof.
\end{proof}

The two stages are illustrated in Figure \ref{fig:upper3} for Type 3. 

\begin{figure}[htbp]
\begin{center}
\setlength{\unitlength}{0.8mm}
\begin{picture}(90,60)(-1,5)
{\color{blue}\qbezier(7,10)(39,15)(77,60)\put(77,60){\line(1,0){3}}\put(65,52){$\phi^{+}$}}
\put(-1,10){\vector(1,0){84}}\put(81,5){$q$}
\put(80,5){\vector(0,1){60}}
\put(49,5){$q^{+}$}
\thicklines
{\color{red}\put(0,10){\line(0,1){50}}
\put(0,10){\line(1,0){50}}
\put(50,10){\line(0,1){20}}
\qbezier(50,30)(65,42)(80,60)}
\put(81.5,58.5){$1$}
{\color{green}\qbezier(10,10)(42,15)(80,60)\put(76,52){$\phi$}}
\end{picture}
\begin{picture}(90,60)(-1,5)
{\color{blue}\qbezier(7,10)(39,15)(77,60)\put(77,60){\line(1,0){3}}\put(65,52){$\phi^{+}$}}
\put(-1,10){\vector(1,0){84}}\put(81,5){$q$}
{\color{cyan}\qbezier(0,10)(50,25)(80,50)\put(70,38){$\stat$}}
\put(80,5){\vector(0,1){60}}
\put(49,5){$q^{+}$}
\multiput(42,10)(0,1){16}{\line(0,1){0.5}}\put(41,5){$\bar q$}
\put(34,10){\line(0,1){10.8}}\put(32.5,5){$\hat q$}
\put(27,10){\line(0,1){6.5}}\put(25,12){\vector(1,0){4}}\put(23,5){$\tilde q(t)$}
\thicklines
{\color{red}\put(0,10){\line(0,1){50}}\put(52,23){$Z(q)$}
\put(0,10){\line(1,0){50}}
\put(50,10){\line(0,1){20}}
\qbezier(50,30)(65,42)(80,60)}
\put(81.5,58.5){$1$}
{\color{green}\qbezier(10,10)(42,15)(80,60)\put(76,52){$\phi$}}
\end{picture}
\caption{Upper envelope, for type 3, stage 1 (left) and stage 2 (right).}
\label{fig:upper3}
\end{center}
\end{figure}
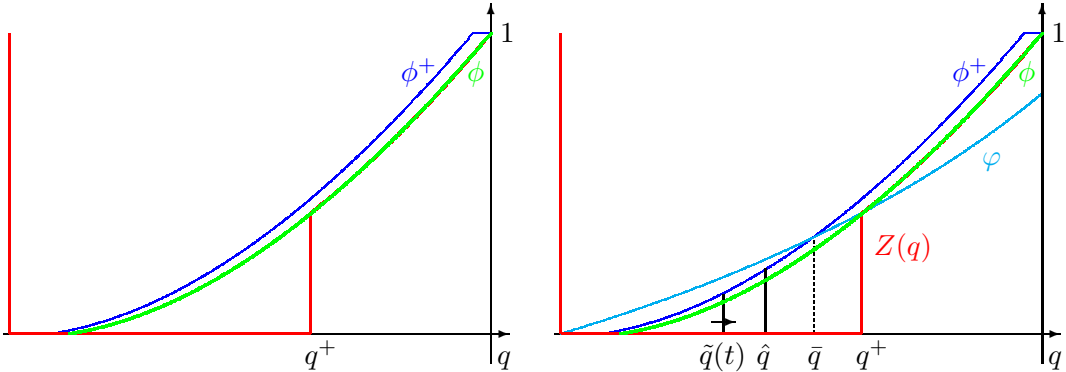

\subsection{Lower envelopes.}

For Type 4, the lower envelope is trivial, since one can simply take
the stationary traveling wave profile $Z(q)$.
We now construct the lower envelopes for Type 1, 2, and 3, 
following  a similar line of arguments as those for the upper envelopes.
In stage 1, we show that the rarefaction fan approaches 
the stationary traveling wave
profile as $t$ grows.

\begin{lemma}\label{lm:LB}
Assume $D<\dss$ and let $\varepsilon>0$ be given. We define the function
\begin{equation}\label{phi-}
 \phi^{-}(q) ~\dot=~  \begin{cases}
   1, & q=-D,\\
   0,& -D<q< q_1, \\
   \phi(q-\varepsilon), \qquad &q_1 \le q \le 0,
 \end{cases}
\end{equation}
 Here the value $q_1$ is determined as follows. 
 If $D$ is sufficiently small such that the graph of $\phi(q-\varepsilon)$ lies completely above the
 graph of $\adm(q+D)$, we let $q_1=-D$, and we remove the second line with $-D<q<q_1$ 
 in  the definition \eqref{phi-}. 
 Otherwise, we let  $q_1$ be the right-most 
 intersection point of those two graphs. 
 Then, there exists a time $\cTa$, such that, 
 \begin{equation}\label{eq:LB}
   \zeta(t,q) \ge \phi^{-}(q), \qquad t\ge \cTa.
 \end{equation}
\end{lemma}

\begin{proof}
The proof follows a similar structure as the one for Lemma \ref{lm:UB}, with some
modifications.
Let $\varepsilon>0$ be given. Since
$\phi^{-}(0)=\phi(-\varepsilon)<1$, by continuity there exists
$q_{o}<0$ such that 
\begin{equation*}
 \zeta(0,q) \ge \phi^{-}(q), \qquad \text{for} \quad q_0 \le q \le 0, \quad 
 \zeta_{o}=\zeta\left(0,q_{o}\right)<1.
\end{equation*}
We consider the rarefaction fronts generated on $[q_{o},0]$ and 
let $t\mapsto q$ be the characteristic curve with $q(0)=q_{o}$.
Then, the point $(q(t),\zeta(t,q(t)))$ 
travels along some shift of the graph of $\phi$ until it merges into a
singularity. There are two possibilities.

(1). If $D$ is small such that the graph of $\phi$ lies completely above the graph
of $\adm(q+D)$, then the characteristic will reach $q=-D$ and enter the downward jump
at $q=-D$.

(2). If the two graphs intersect, then the right-most position of the
right front of a possible shock is $q_1$, and \eqref{eq:LB} holds. 
If the actual shock front is  to the left of $q_1$, or the characteristic
does not enter any shock,  we still have \eqref{eq:LB}. 

By Lemma \ref{lm:4a} we have 
$\cTa=T_{\zeta_{o}}$, which is finite because $\zeta_{o}<1$, completing the proof.
\end{proof}

\bigskip

If $D$ is small such that $q_1=-D$ in Lemma \ref{lm:LB}, the lower envelopes are complete. 
For the rest of the subsection, we assume $q_1 > -D$.
We now enter stage 2, and we control the location of the right front of the shock.
We will use again the symbols $\bar q, \tilde q, \hat q, \bar C$ etc, 
for different values, without causing any confusion.

\begin{lemma}\label{lm:LB2}
  Assume $D<\dss$ and $q_1>-D$ in \eqref{phi-}. 
  Let $q^{+}$ be the location of right front of the stationary shock for Type 3, 
  and let $q^{+}=-D$ if it is Type 1 and 2. 
  Let $\varepsilon>0$ be given. There exists a time $\cTb$,
  such that 
  \begin{equation*} 
  \zeta(t,q) \ge Z^{-}(q), \qquad -D\le q\le 0, \quad
  \text{for} \quad t \ge \cTa+\cTb.
  \end{equation*}
  Here the function $Z^{-}$ is defined as follows.  For Type 1, we let
  \[ 
   Z^{-}(q) ~\dot=~  \begin{cases}
   1, & q=-D,\\
   \phi(q-\varepsilon), \qquad & -D < q \le 0.
   \end{cases}
  \] 
  For Type 2 and 3, we let
  \begin{equation}\label{Z-}
   Z^{-}(q) ~\dot=~  \begin{cases}
   1, & q=-D,\\
   0,& -D<q\le \hat q , \qquad \hat q= q^{+}+C_2\varepsilon , \\
   \phi(q-\varepsilon), \qquad & \hat q < q \le 0.
   \end{cases}
  \end{equation}
  The constant $C_2$ does not depend on $\varepsilon$. 
\end{lemma}

\begin{proof}
Again, the proof follows a similar line of arguments as for Lemma \ref{lm:UB2},
with modifications. 
Let $\bar q$ be the intersection point of the graphs of $\phi(q-\varepsilon)$ 
and $\stat(q+D)$. We have
\begin{equation*}
\bar q-q^{+} \le \bar C \varepsilon, \qquad 
\bar C =\frac{1}{\kappa} \cdot \max_{0<z<1} \frac{h^2(z)}{h'(z)}\,.
\end{equation*}
We can now choose the constant in \eqref{Z-} as $C_2=2 \bar C$. 

The lower envelopes are defined as follows, for $t\ge \cTa$,
\[ 
  \mathcal{Z}^{-}(t,q) ~\dot=~ \begin{cases}
  1, \qquad & q=-D,\\
  0, \qquad & -D<q\le \tilde q(t) ,\\
  \phi(q-\varepsilon), \qquad &\tilde q(t) < q \le 0.
  \end{cases}
\] 
Here the front $\tilde q(t)$ travels with the speed as if it were 
the right front of a shock. By \eqref{eq:srv}, 
$\tilde q(t)$ satisfies the ODE, for $t\ge \cTa$
\begin{equation}\label{tqODE-}
  \tilde q'(t) = F(\zeta;\tilde q) \cdot \frac{1-\phi(\tilde q-\varepsilon)}{\phi(\tilde q-\varepsilon)}
  \left[\psi(\tilde q+D)-h(\phi(\tilde q-\varepsilon))\right], 
  \qquad \tilde q(\cTa) = q_1\,.
\end{equation}
The ODE \eqref{tqODE-} is solved for $t\ge \cTa$, 
until at some time $t=\cTa+\cTb$ when
the front $\tilde q(t)$
passes the one in $Z^{-}$, i.e., when 
$\tilde q(\cTa+\cTb) \le \hat q$. 

We now show that $\cTb$ is finite for any given $\varepsilon$.
As in the proof of Lemma \ref{lm:UB2},
observe that when $\tilde q(t)> \bar q$ the graph of $\phi^{-}$ lies 
strictly above the 
graph of $\stat(q+D)$, therefore, by continuity and since
$q^{-}+C_{2}\varepsilon > \bar q$ 
we have
\begin{displaymath}
  v_{\varepsilon}=-\max\left\{\stat\left(q+D\right)-\phi^{-}(q):\quad
    q\ge q^{+}+C_{2}\varepsilon\right\}>0.
\end{displaymath}
Hence, as long as $q^{+}+C_{2}\varepsilon \le \tilde q(t)\le q_{1}$, we can compute 
\[ 
  \begin{split}
    \tilde q'(t) &\le \left[1-\phi^{-}(\tilde q)\right] \left[\psi(\tilde
      q+D)-h(\phi^{+}(\tilde q))\right]\\
    &\le \left[1-\phi(q_{1} )\right]\left[h\left(\stat\left(\tilde q +
          D\right)\right)-h(\phi^{+}(\tilde q))\right]\\
    &\le -\left[1-\phi( q_{1})\right]\cdot\min h' \cdot v_{\varepsilon}.
  \end{split}
\] 
We can now conclude that 
\[ 
\mathcal{Z}^{-} (t,q) \ge Z^{-}(q), \qquad \text{for}\quad t\ge \cTa+\cTb\,,
\quad \text{where}\quad 
\cTb= \frac{\dhk + \varepsilon}{\left[1-\phi( q_{1})\right]\cdot\min h' \cdot v_{\varepsilon}}\,.
\] 

We note that for Type 1, the front $\tilde q(t)$ would merge into $q=-D$ 
at some $t<\cTa+\cTb$.

It remains to show that the solution  $\zeta(t,q)$ satisfies
\begin{equation}\label{comp-}
\zeta(t,q)\ge \mathcal{Z}^{-}(t,q)\qquad
\text{for} \quad  
\cTa \le t 
\le \cTa+\cTb\,.
\end{equation} 
Indeed, by construction  \eqref{comp-} holds for $t=\cTa$, since 
$\mathcal{Z}^{-}(\cTa,q)=\phi^{-}(q)$,
where $\phi^{-}$ is the the function defined in Lemma \ref{lm:LB}. 
Now we consider a later time $\bar t\ge \cTa$.
By Lemma \ref{lm:4a}, the smooth part of the solution $\zeta(\bar t,q)$
remains above the graph of $\phi(q-\varepsilon)$. 
It is then enough to check the speed of right front of the possible shock.
Assume that $z(\bar t,q)$ has a shock, with its right front at $\check q(t)$, and
\begin{equation*}
\check q(\bar t)=\tilde q(\bar t), \qquad 
\zeta(\bar t, \check q(\bar t)) \ge \phi(\tilde q(\bar t)-\varepsilon).
\end{equation*}
By \eqref{eq:srv} we clearly have 
$
 \check q'(\bar t) \le \tilde q'(\bar t)
$,
so the graph of $\zeta(t,q)$ remains above that of $\mathcal{Z}^{-}(t,q)$, 
completing the proof.
\end{proof}

The two stages are illustrated in Figure \ref{fig:upper3} for Type 3. 

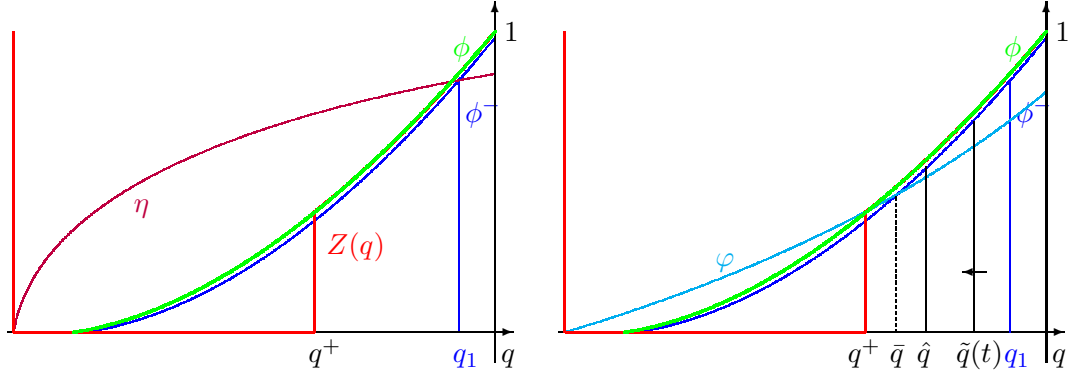
\begin{figure}[htbp]
\begin{center}
\setlength{\unitlength}{0.8mm}
\begin{picture}(90,60)(-1,5)
{\color{blue}\put(74,10){\line(0,1){42}}\put(73,5){$q_1$}
\qbezier(12,10)(44,15)(80,59)\put(75,45){$\phi^{-}$}}
\put(-1,10){\vector(1,0){84}}\put(81,5){$q$}
{\color{purple}\qbezier(0,10)(5,40)(80,53)\put(20,30){$\adm$}}
\put(80,5){\vector(0,1){60}}
\put(49,5){$q^{+}$}
\thicklines
{\color{red}\put(0,10){\line(0,1){50}}\put(52,23){$Z(q)$}
\put(0,10){\line(1,0){50}}
\put(50,10){\line(0,1){20}}
\qbezier(50,30)(65,42)(80,60)}
\put(81.5,58.5){$1$}
{\color{green}\qbezier(10,10)(42,15)(80,60)\put(73,56){$\phi$}}
\end{picture}
\begin{picture}(90,60)(-1,5)
\put(-1,10){\vector(1,0){84}}\put(81,5){$q$}
{\color{blue}\put(74,10){\line(0,1){42}}\put(73,5){$q_1$}
\qbezier(12,10)(44,15)(80,59)\put(75,45){$\phi^{-}$}}
\put(47,5){$q^{+}$}
{\color{cyan}\qbezier(0,10)(50,25)(80,50)\put(25,21){$\stat$}}
\put(80,5){\vector(0,1){60}}
\multiput(55,10)(0,1){23}{\line(0,1){0.5}}\put(54,5){$\bar q$}
\put(60,10){\line(0,1){27.6}}\put(58.5,5){$\hat q$}
\put(68,10){\line(0,1){35}}\put(70,20){\vector(-1,0){4}}\put(65,5){$\tilde q(t)$}
\thicklines
{\color{red}\put(0,10){\line(0,1){50}}
\put(0,10){\line(1,0){50}}
\put(50,10){\line(0,1){20}}
\qbezier(50,30)(65,42)(80,60)}
\put(81.5,58.5){$1$}
{\color{green}\qbezier(10,10)(42,15)(80,60)\put(73,56){$\phi$}}
\end{picture}
\caption{Lower envelope, for Type 3, with Stage 1 (left) and Stage 2 (right).}
\label{fig:lower3}
\end{center}
\end{figure}

\bigskip

We now combine all the estimates and prove Theorem \ref{thm:stab}.

\begin{proof} (of Theorem \ref{thm:stab}) 
Given $\varepsilon>0$. By Lemma \ref{lm:UB2} and Lemma \ref{lm:LB2}, 
there exists a time 
\begin{equation*}
T^\varepsilon = \max \left\{ \Ta+\Tb \,, \, \cTa+\cTb \right\},
\end{equation*}
such that for $t\ge T^\varepsilon $, we have
\begin{equation}\label{eq:squeeze}
 Z^{-}(q) \le \zeta(t,q) \le Z^{+}(q), \qquad -D\le q \le 0\,.
\end{equation}
Here $Z^{-}$ and $Z^{+}$ are defined in \eqref{Z-} and \eqref{Z+}, respectively, and they satisfy
\begin{equation}\label{eq:squeeze2}
 Z^{-}(q) \le Z(q) \le Z^{+}(q), \qquad -D\le q \le 0\,,
\end{equation}
\begin{equation}\label{eq:squeeze3}
\left\| Z^{-} - Z\right\|_{\mathbf{L}^1} \le C_2 \varepsilon, \qquad 
\left\| Z^{+} - Z\right\|_{\mathbf{L}^1} \le C_1 \varepsilon, \qquad
\left\| Z^{+} - Z^-\right\|_{\mathbf{L}^1} \le (C_1+C_2) \varepsilon,
\end{equation}
where the constants $C_1,C_2$ does not depend of $\varepsilon$. 
Thanks to \eqref{eq:squeeze}-\eqref{eq:squeeze3}, we now conclude
\begin{equation*}
 \left\| \zeta(t,\cdot) - Z(\cdot)\right\|_{\mathbf{L}^1} \le
 \left\| Z^{+} - Z^-\right\|_{\mathbf{L}^1} 
 \le (C_1+C_2) \varepsilon,\qquad 
 t\ge T^\varepsilon\,,
\end{equation*}
completing the proof.
\end{proof}

\section{A Numerical Example}
In this section we present a numerical simulation of \eqref{1.3}.
We generate piecewise constant approximate solutions 
using an extended version of the wave front tracking algorithm 
described in \cite{CGS,ShZh}.
We use the erosion function
$$g(z)=\left(1-z\right)\left(\frac{1}{2}+z\right)$$ 
and the initial data
\begin{displaymath}
  z_{o}(u)=
  \begin{cases}
    1&\text{ for }u\le 0\,,\\
    0&\text{ for }0<u\le 0.6\,,\\
    \exp\left(-\frac{1}{2}\left(u+0.11\right)\right)&\text{ for }0.6<u\,.
  \end{cases}
\end{displaymath}

With this initial data, the traveling wave profile is of Type 3. 
Solutions are plotted for nine  different values of $t$,  
for both the functions $\zeta(t,q)$ and $z(t,u)$, in 
Figure \ref{fig:qzgraphics}  and Figure \ref{fig:uzgraphics}, respectively. 
In Figure \ref{fig:qzgraphics} we also plotted the 
stationary Type 3 traveling wave $Z(q)$ (in red) together with the 
solution $\zeta(t,q)$. 

As we observe in Figure \ref{fig:uzgraphics},
the traveling waves for the solutions $z(t,u)$ are not stationary. 
We clearly observe the moving traveling wave in the 
last 2-3 plots in Figure \ref{fig:uzgraphics}.
It is interesting to observe that, 
for waves of Type 1, 2 and 3, they  all travel 
with the same constant speed $\tilde \sigma$
which can be deduced {from} \eqref{sigma}, i.e.,
\begin{displaymath}
  \tilde\sigma = \frac{1}{\sigma}=\frac{1}{f'(1)}= -g'(1)=\frac{3}{2}=1.5.
\end{displaymath}
For a Type 4 wave, it travels with the shock speed. 

This simulation also demonstrates the complexity of the transient dynamics
of the wave formation and interaction.
One obverses that the shock in the initial data 
disappears as the rarefaction wave on the right merges into the shock,  
only to reappears later as a new shock forms.

\begin{figure}[H]
  \centering
  \includegraphics[width=0.36\linewidth]{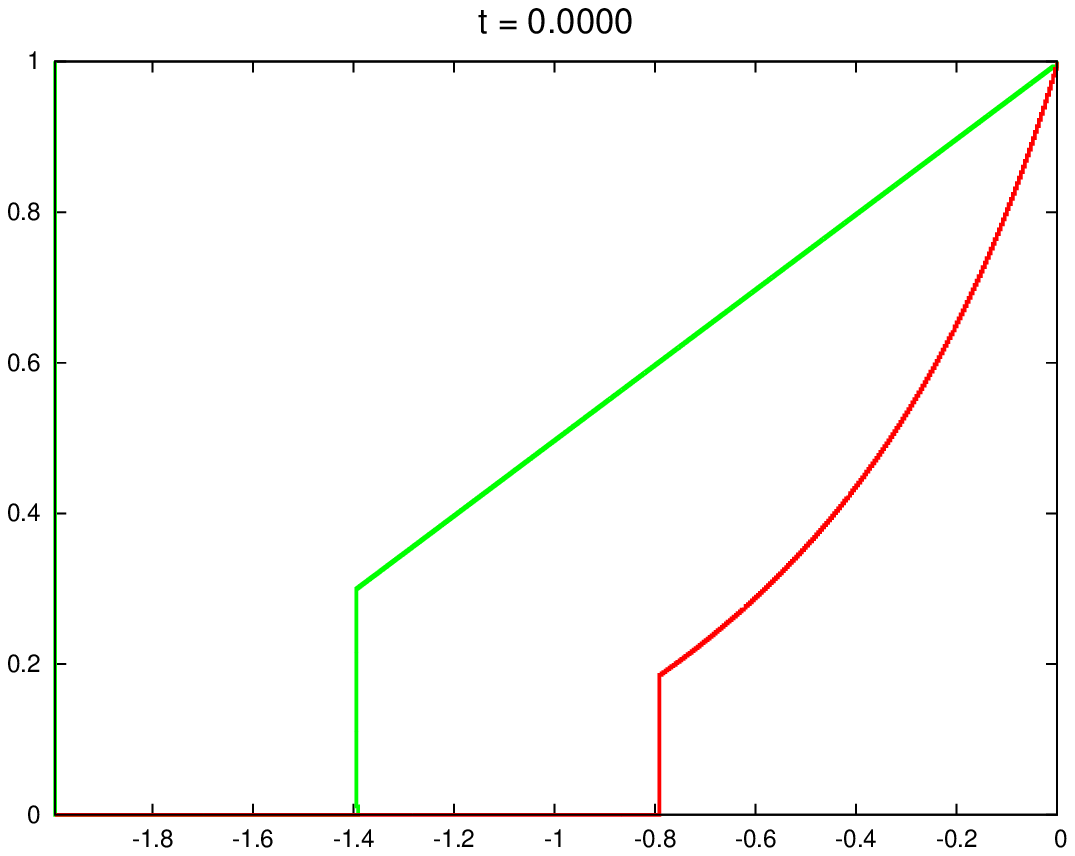}\!\!\!\!\!\!\!\!\!\!\!\!\!
  \includegraphics[width=0.36\linewidth]{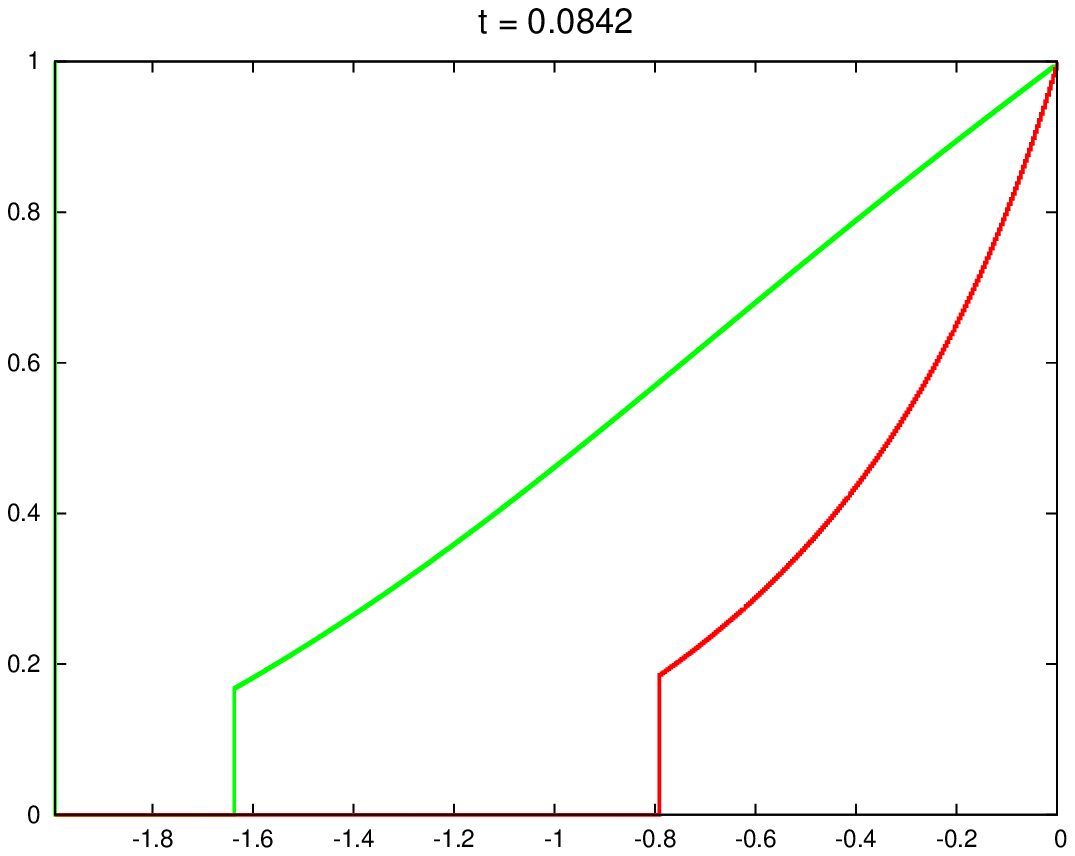}\!\!\!\!\!\!\!\!\!\!\!\!\!
  \includegraphics[width=0.36\linewidth]{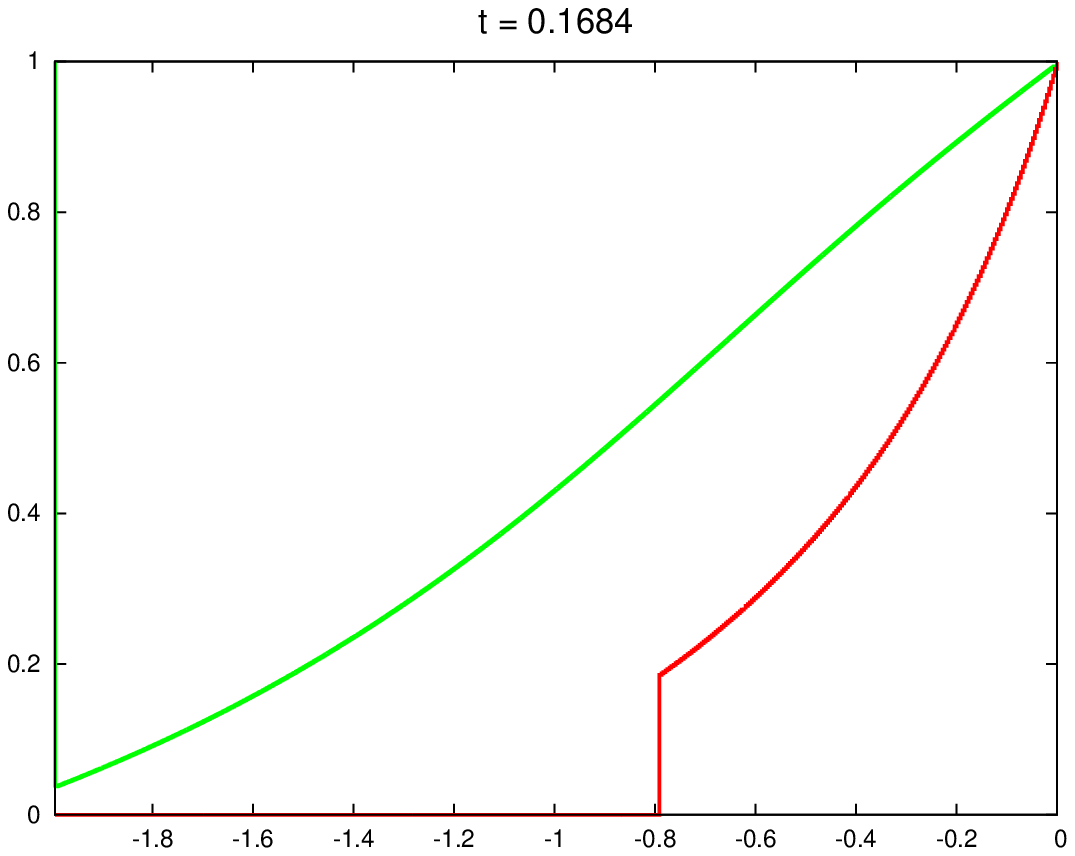}\\
  \includegraphics[width=0.36\linewidth]{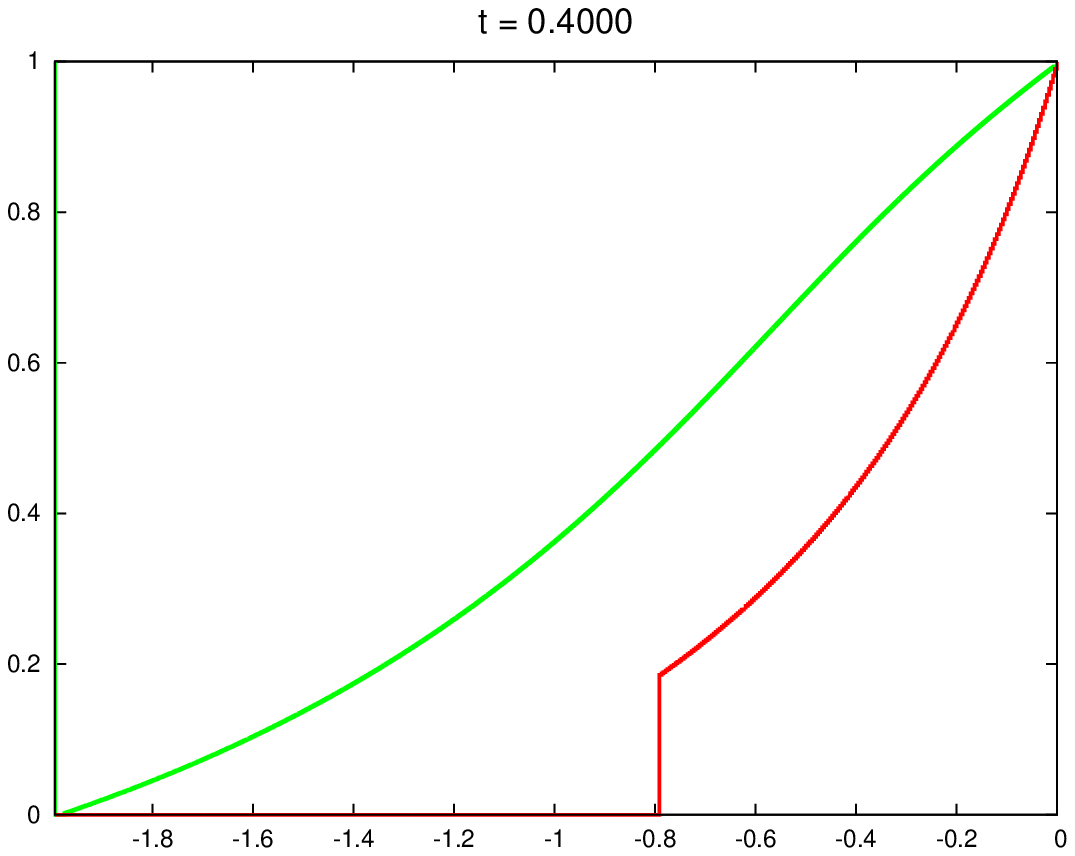}\!\!\!\!\!\!\!\!\!\!\!\!\!
  \includegraphics[width=0.36\linewidth]{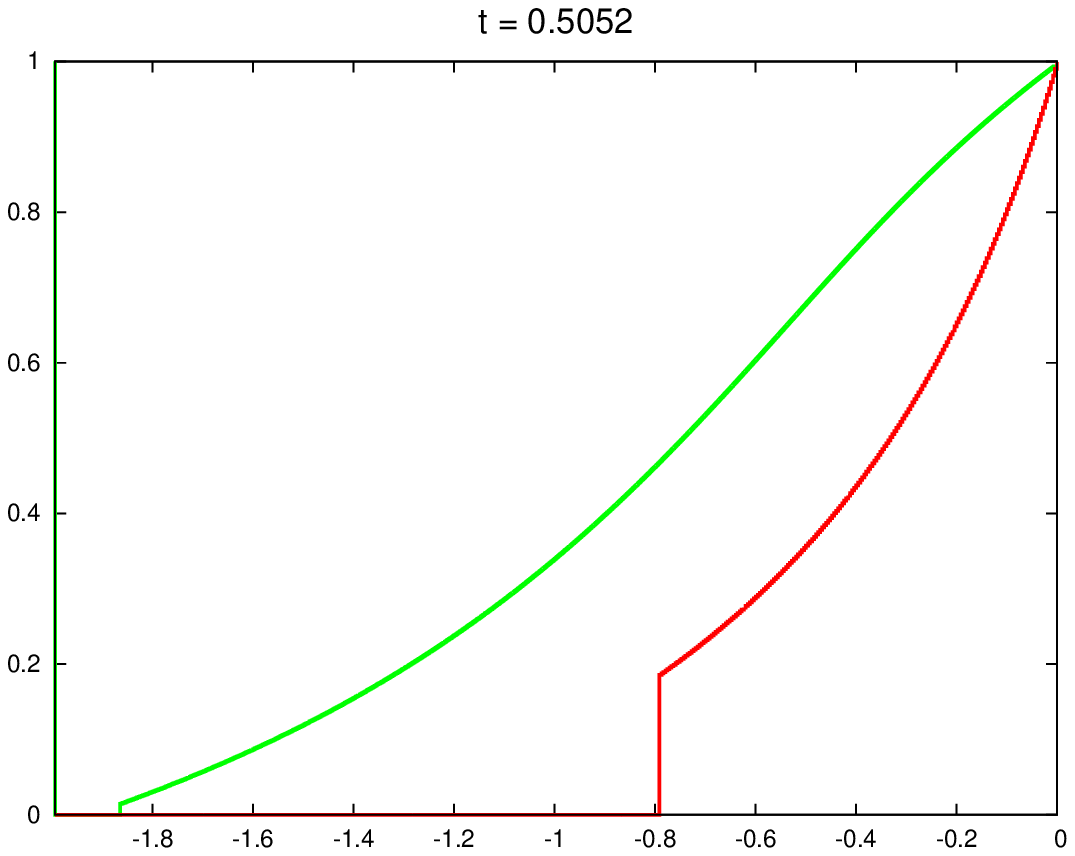}\!\!\!\!\!\!\!\!\!\!\!\!\!
  \includegraphics[width=0.36\linewidth]{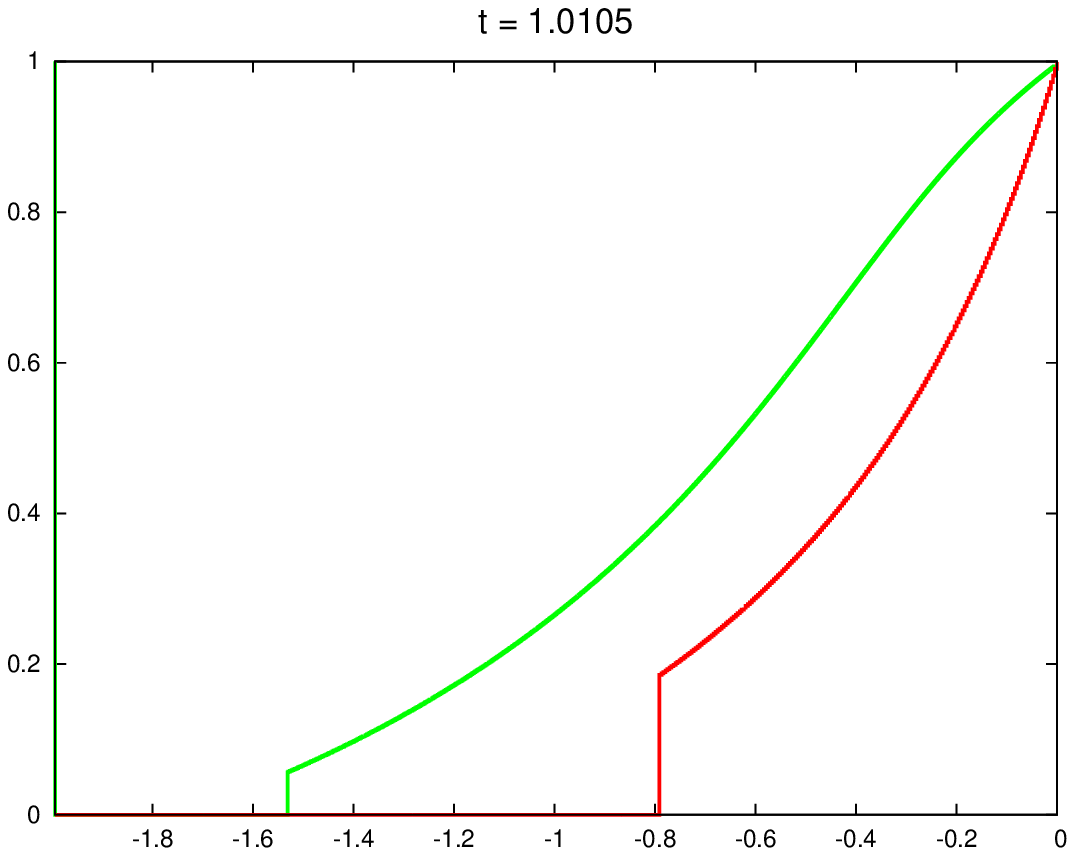}\\
  \includegraphics[width=0.36\linewidth]{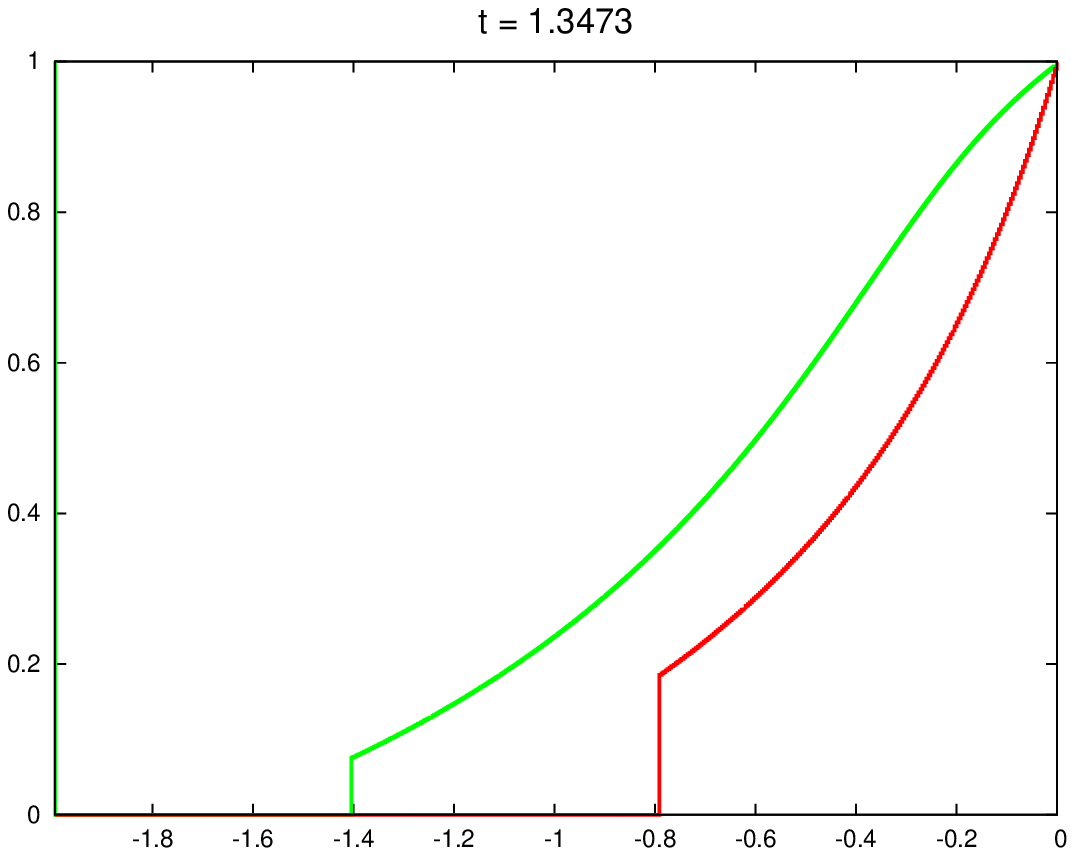}\!\!\!\!\!\!\!\!\!\!\!\!\!
  \includegraphics[width=0.36\linewidth]{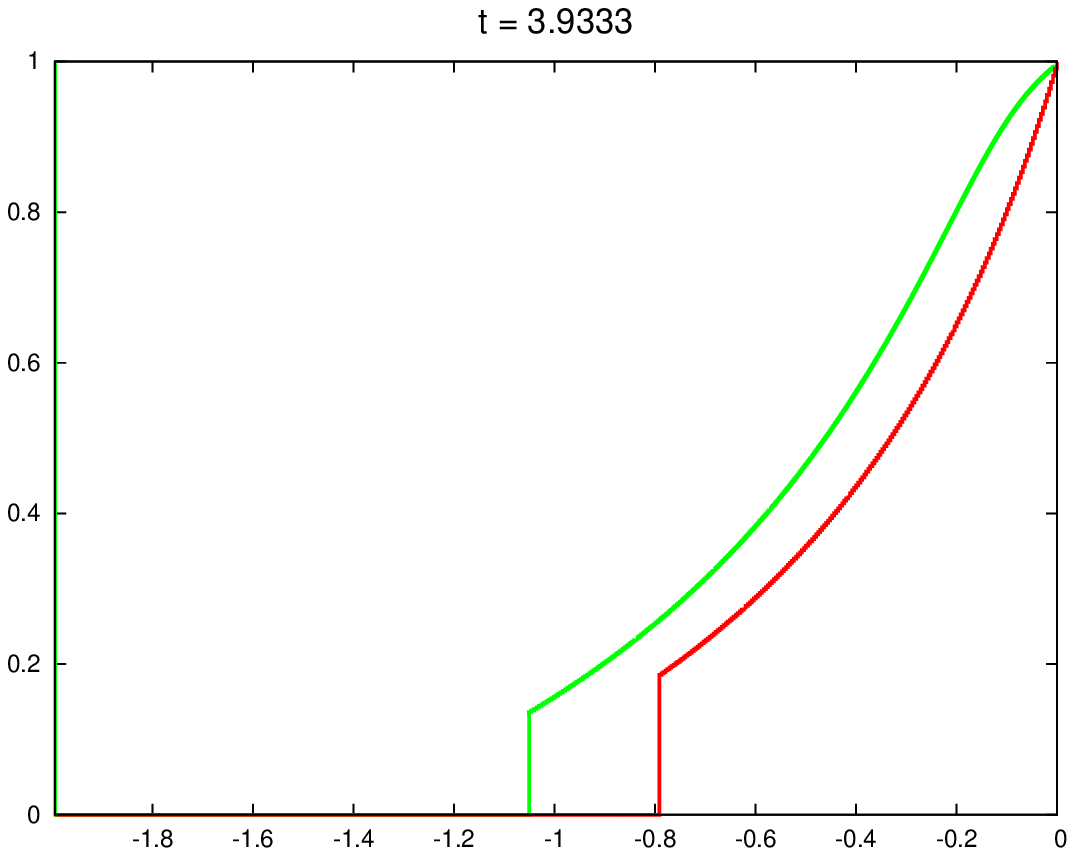}\!\!\!\!\!\!\!\!\!\!\!\!\!
  \includegraphics[width=0.36\linewidth]{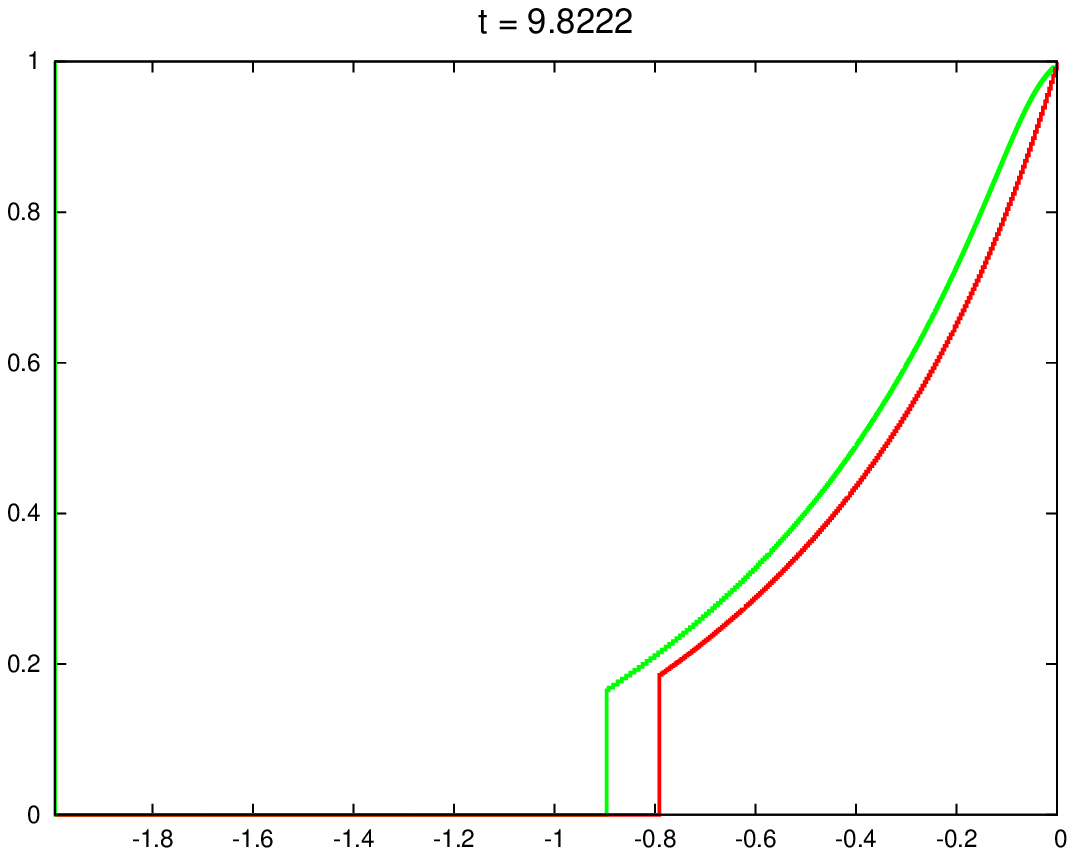}\\
  \Caption{Plots for the function $\zeta(t,q)$ vs $q$ for various values of $t$.}
  \label{fig:qzgraphics}
\end{figure}
\begin{figure}[H]
  \centering
  \includegraphics[width=0.36\linewidth]{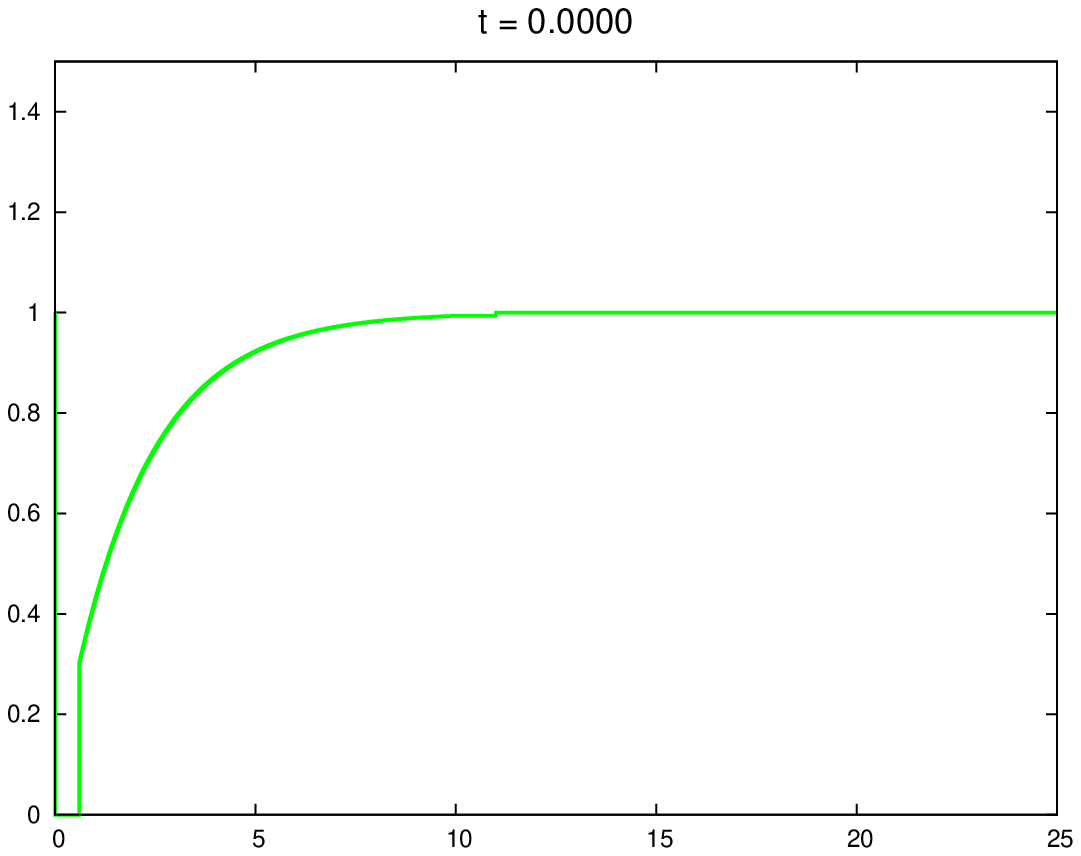}\!\!\!\!\!\!\!\!\!\!\!\!\!
  \includegraphics[width=0.36\linewidth]{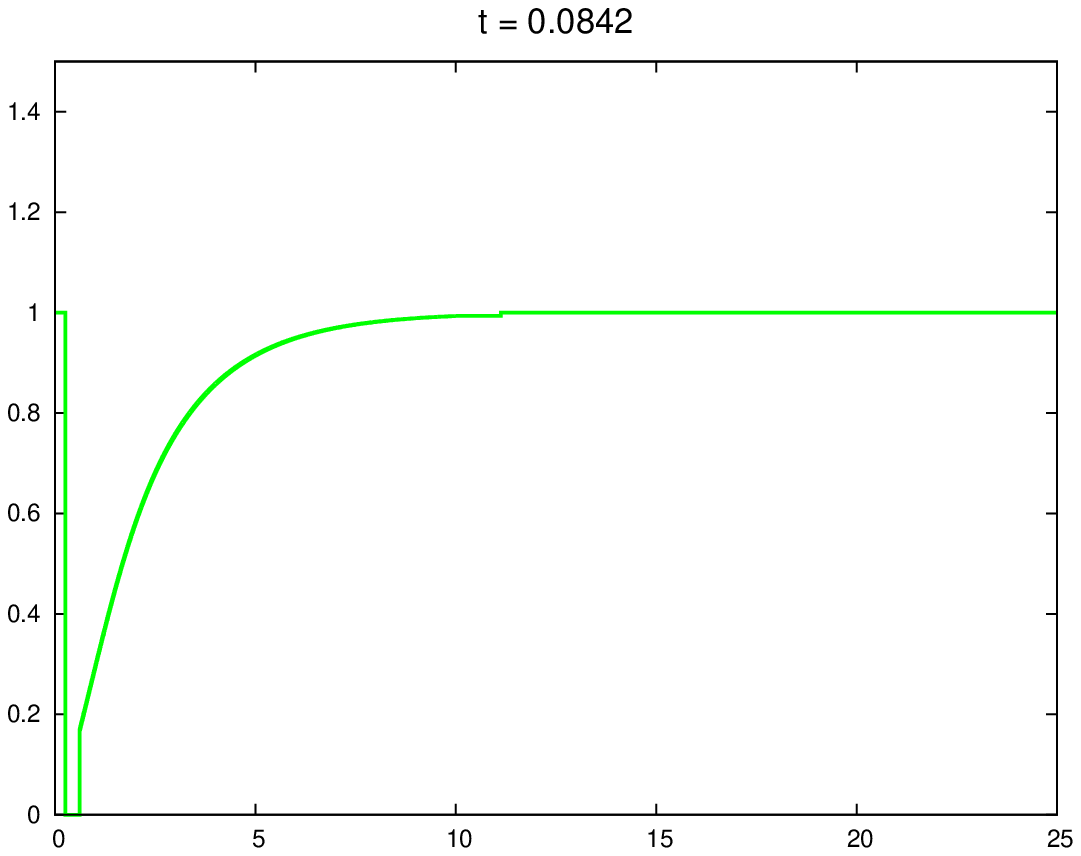}\!\!\!\!\!\!\!\!\!\!\!\!\!
  \includegraphics[width=0.36\linewidth]{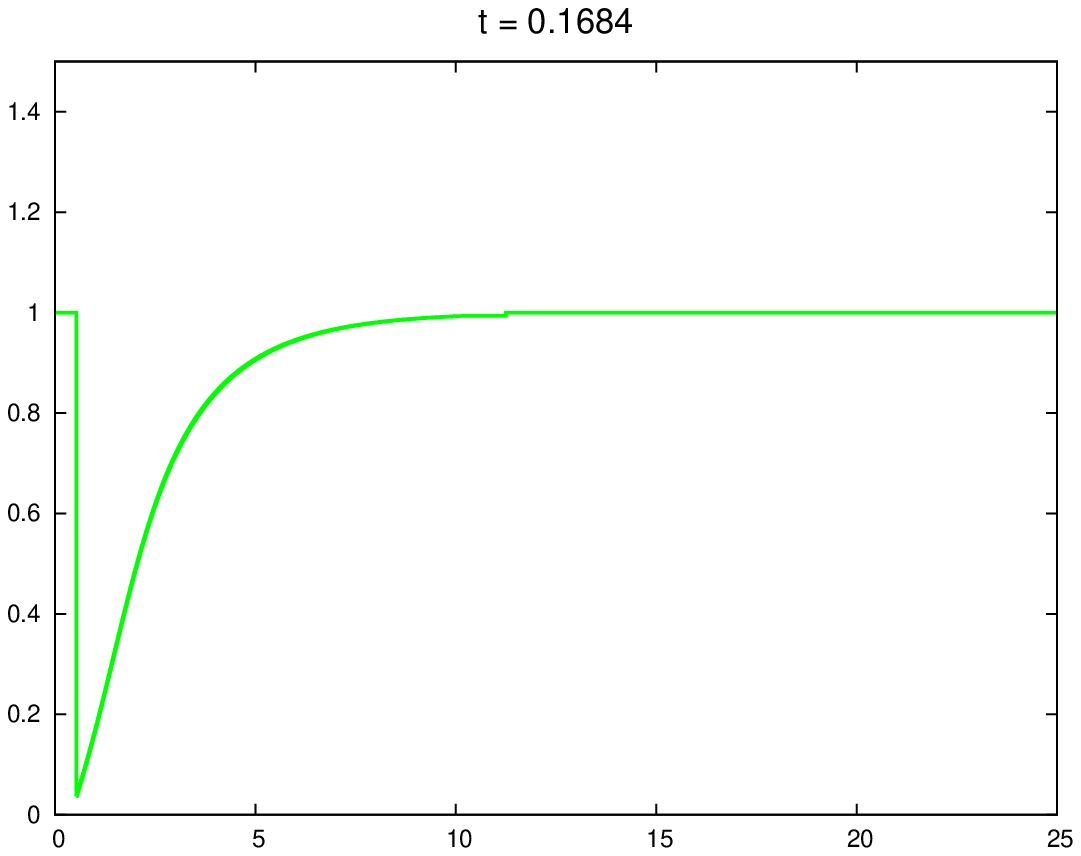}\\
  \includegraphics[width=0.36\linewidth]{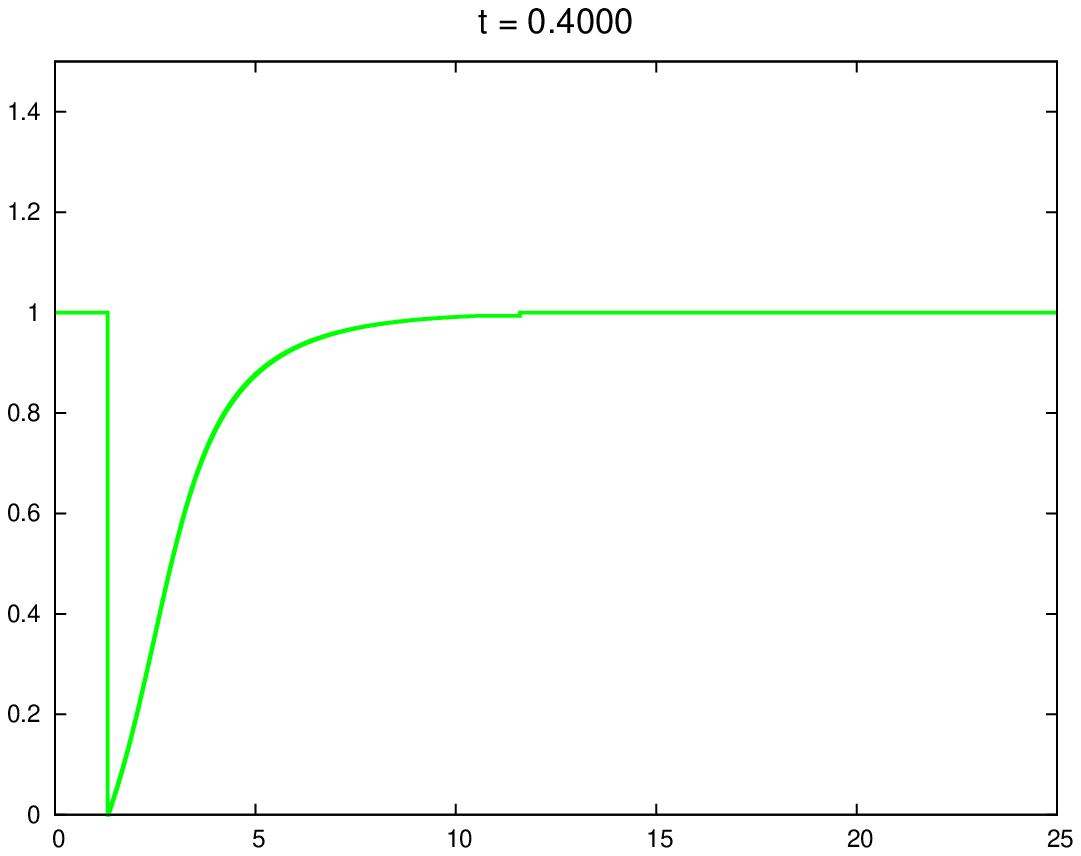}\!\!\!\!\!\!\!\!\!\!\!\!\!
  \includegraphics[width=0.36\linewidth]{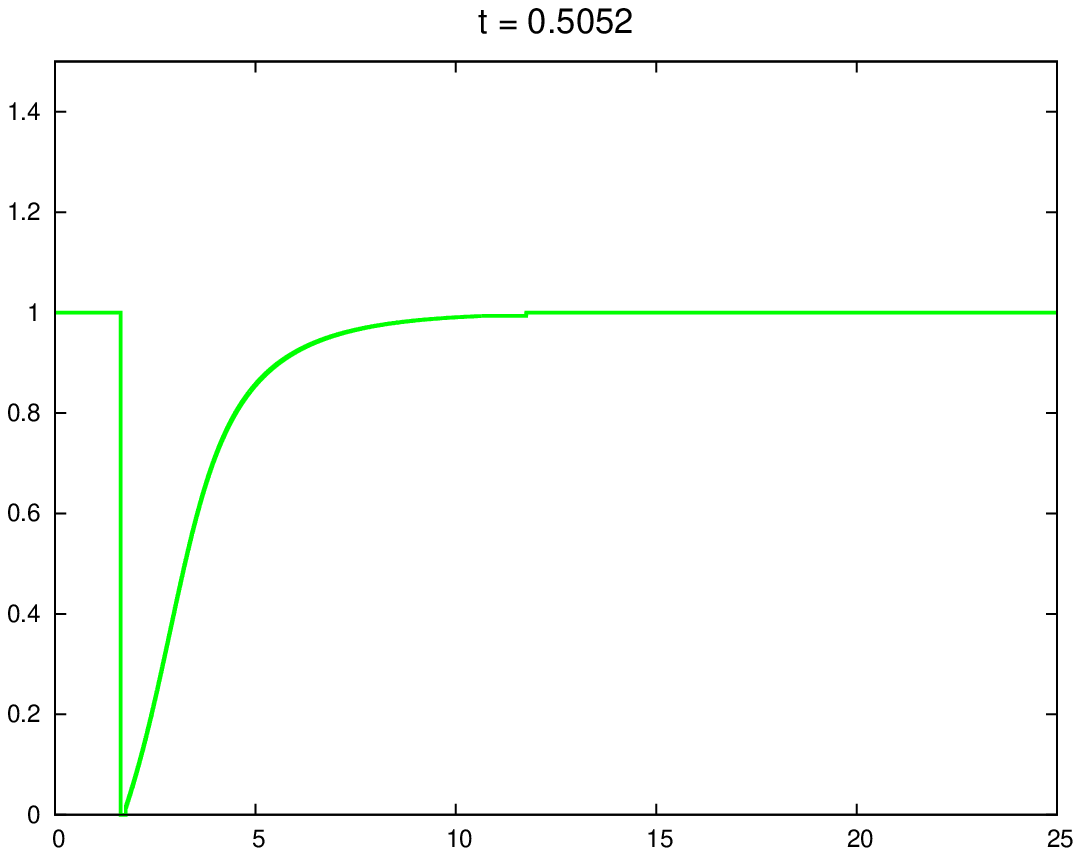}\!\!\!\!\!\!\!\!\!\!\!\!\!
  \includegraphics[width=0.36\linewidth]{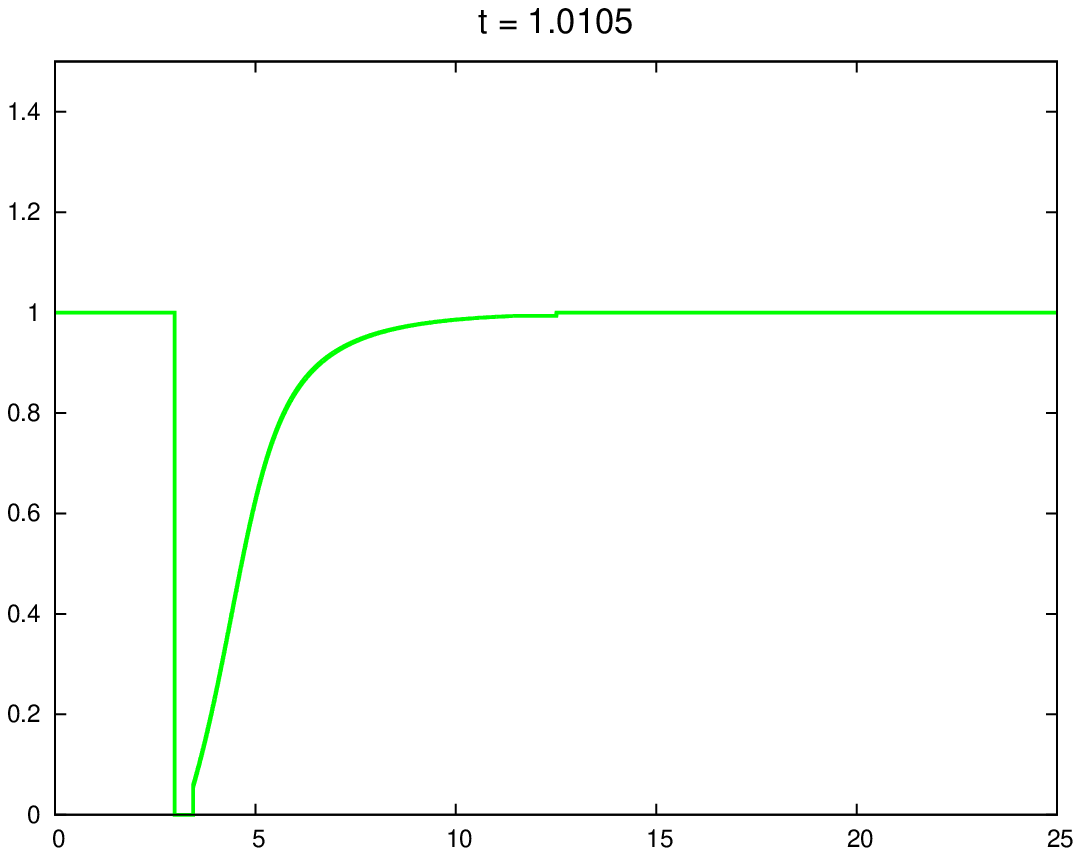}\\
  \includegraphics[width=0.36\linewidth]{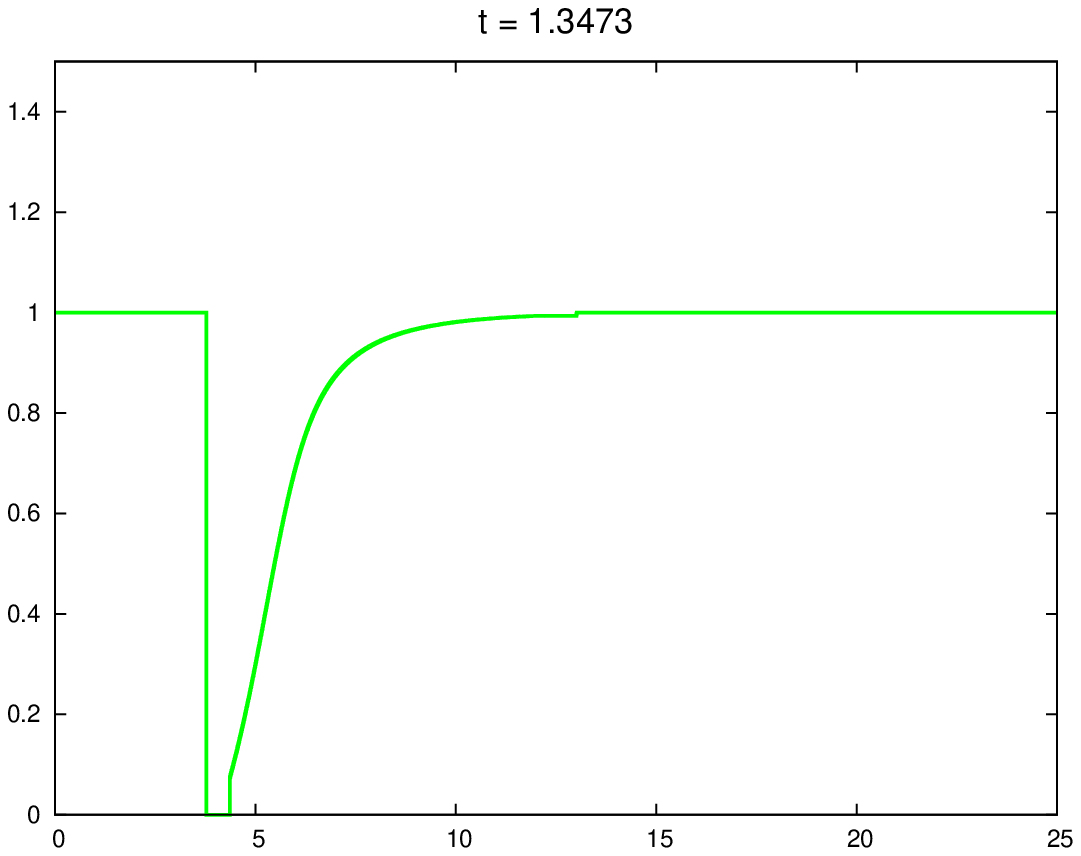}\!\!\!\!\!\!\!\!\!\!\!\!\!
  \includegraphics[width=0.36\linewidth]{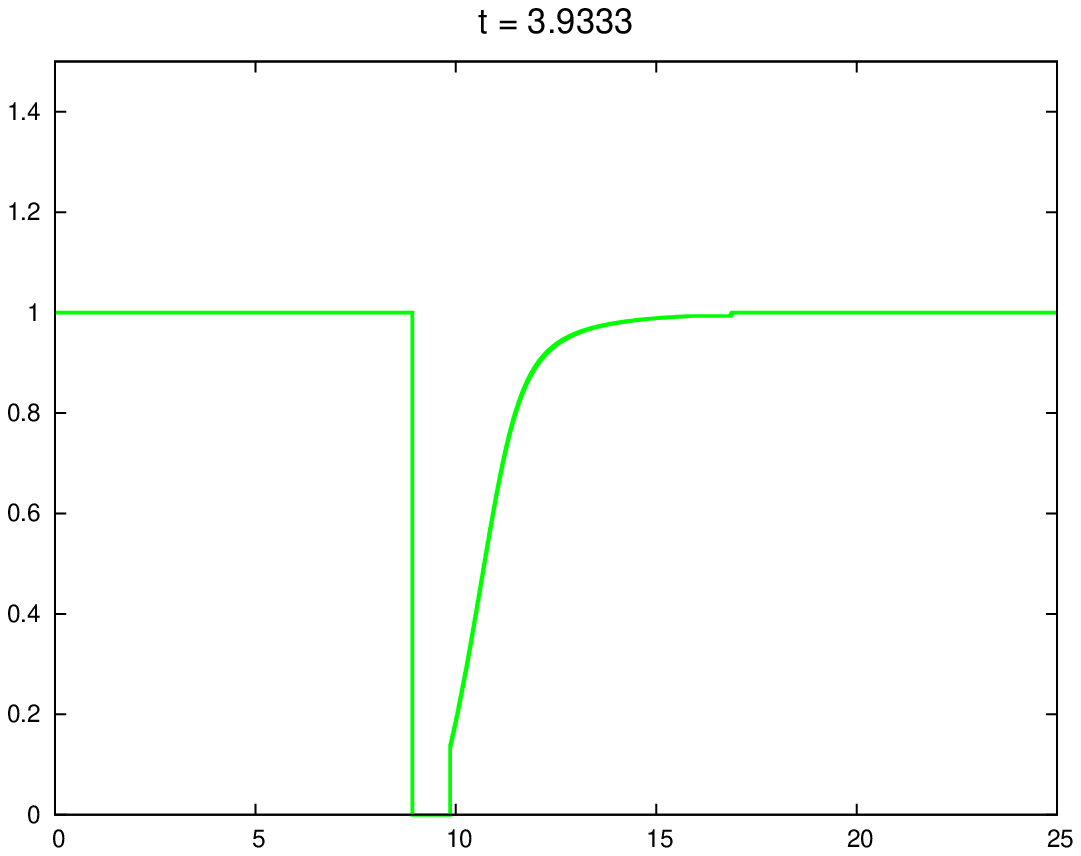}\!\!\!\!\!\!\!\!\!\!\!\!\!
  \includegraphics[width=0.36\linewidth]{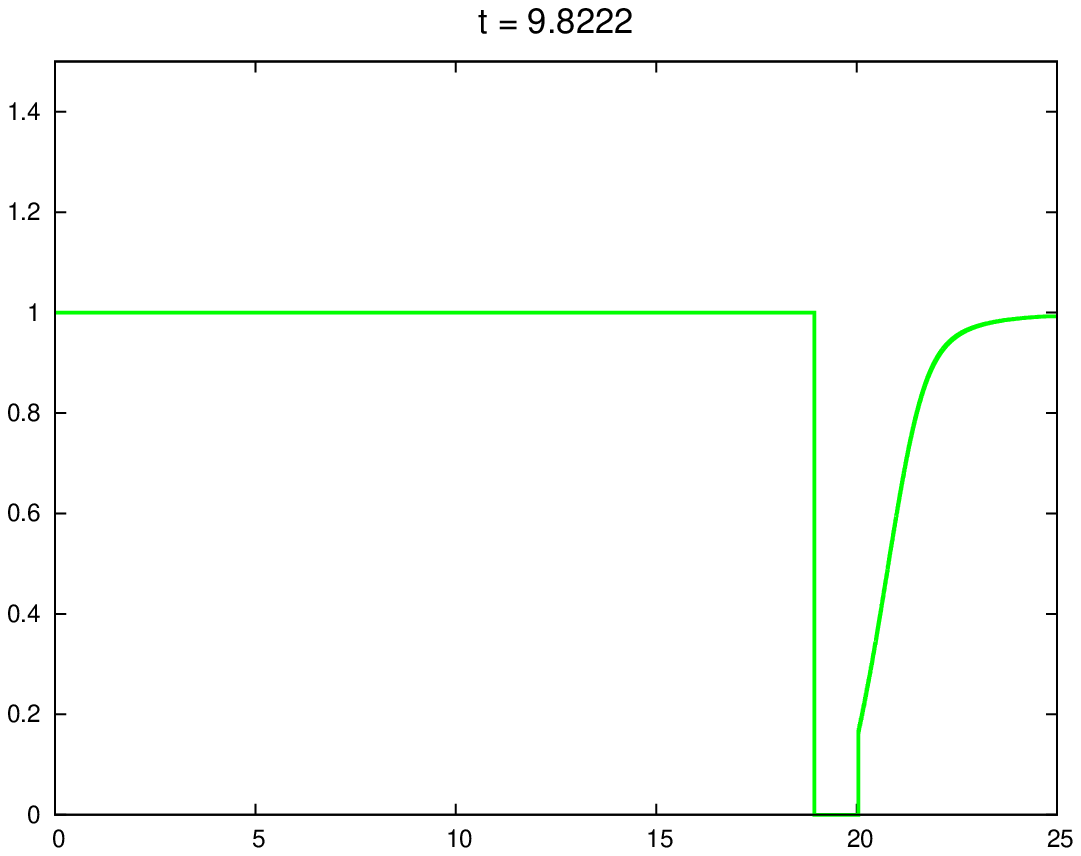}\\
  \Caption{Plots for the function $z(t,u)$ vs $u$ for various values of $t$.}
  \label{fig:uzgraphics}
\end{figure}


\section{Concluding remarks}

In this paper we prove the existence of traveling wave profiles for an 
integro differential equation modeling slow erosion of granular flow.
Such profiles are unique with respect to the total drop $D$.
Furthermore we show that these profiles provide local attractors 
for the solutions of the Cauchy problem.

We now conclude the paper with several final remarks.

\begin{remark}\label{rmk1}
The basin of attraction of the traveling wave profile is actually much larger,
and the initial data does not need to be non-decreasing. 
The initial data $\zeta_0(q)=\zeta(0,q)$ only needs to satisfy the following.
For some $\epsilon>0$, we have
\begin{eqnarray}
  \zeta_0(-D)=\zeta_0(0)=1,\quad
  \zeta_0(q) \le 1- C\epsilon, && (-D+\epsilon\le q\le  -\epsilon) \label{cr1}\\
  \zeta_0(q_1)-\zeta_0(q_2) \ge 0, && (-\epsilon \le q_2<q_1<0)\label{cr2}\\
  \zeta_0(q_4)-\zeta_0(q_3) \ge 0, && (-D \le q_4<q_3< -D+\epsilon) \label{cr3}\\
  \text{TV}\{\zeta_0\} \le M,\quad
  \| \zeta_0-1 \|_{\mathbf{L}^1} \le M.&&\label{cr4}
\end{eqnarray}

{From} \cite{CGS} we see that for general initial data with bounded variation, 
the total variation of $\zeta(t,\cdot)$ can grow exponentially in $t$. 
However for this simpler case \eqref{cr1}-\eqref{cr4}, 
one should be able to improve the BV estimate and 
actually obtain a bound that is uniform in $t$. 
We now provide a formal argument. 
Consider a characteristic curve $t\mapsto q(t)$ 
initiated on the interval $[-D+\epsilon,-\epsilon]$. 
As long as $\zeta(t,q(t))>0$,
we have $\dot q <-c_0 \epsilon$ and $\dot \zeta(t,q(t))<-c_0 \epsilon$,
i.e., the characteristic curve travels strictly to the left,
and the $\zeta$ value is strictly decreasing along the characteristics.
In finite time, this curve will either enter a shock such that $\zeta=0$, 
or reach $q=-D$. 
This implies that  
all the singularities would finish all possible interactions in finite time, 
and after that 
the solution $\zeta(t,q)$ will be non-decreasing, as in the assumption \eqref{z02}.
Therefore, after finite time 
the total variation of $\zeta(t,\cdot)$ will be bounded by $2$.
Then one can apply the result in this paper and obtain the asymptotic behavior.
\end{remark}

\begin{remark}\label{rmk2}
At this point, we also conjecture that our result could be extended to 
general  BV initial data $z_{o}(u)$, provided that the total drop $D$ is positive.
Due to the nonlinearity of the erosion function,
all singularities will interact and merge into a single singularity in finite time, 
even though the transient dynamics could be very complicated. 
The solution will satisfy the assumption \eqref{z02} in finite time.
It should be possible to carry out 
a rigorous analysis
through piecewise constant approximate solutions generated by the 
front tracking algorithm.
\end{remark}

\bibliographystyle{amsplain}

\providecommand{\bysame}{\leavevmode\hbox to3em{\hrulefill}\thinspace}
\providecommand{\MR}{\relax\ifhmode\unskip\space\fi MR }
\providecommand{\MRhref}[2]{%
  \href{http://www.ams.org/mathscinet-getitem?mr=#1}{#2}
}
\providecommand{\href}[2]{#2}

\end{document}